\documentclass[12pt,a4paper]{amsart}

\usepackage[left=3cm,right=3cm,top=3.5cm,bottom=3.5cm]{geometry}

\usepackage{amsmath,amsthm,amssymb}
\usepackage{mathrsfs}
\usepackage{lmodern}
\usepackage[T1]{fontenc}
\usepackage[all]{xy}
\usepackage{dsfont}
\usepackage[english]{babel}

\usepackage{multirow}
\usepackage{enumerate}

\usepackage{tikz,pgffor}

\newtheorem{theorem}{Theorem}[section]

\newtheorem{prop}[theorem]{Proposition}

\newtheorem{cor}[theorem]{Corollary}

\theoremstyle{definition}
\newtheorem{definition}[theorem]{Definition}
\newtheorem{remark}[theorem]{Remark}
\newtheorem{example}[theorem]{Example}

\numberwithin{equation}{section}

\newcommand{\rleft}{\mathopen{}\mathclose\bgroup\left}
\newcommand{\rright}{\aftergroup\egroup\right}

\newcommand{\Cd}{\mathds{C}}
\newcommand{\Qd}{\mathds{Q}}

\def\T{{\Bbb T}}
\def\C{{\Bbb C}}

\def\N{{\Bbb N}}
\def\Q{{\Bbb Q}}
\def\R{{\Bbb R}}
\def\P{{\Bbb P}}
\def\Z{{\Bbb Z}}

\begin{document}

\title[The stringy Euler number
and the Mavlyutov duality]{The stringy Euler number of
Calabi-Yau hypersurfaces in toric varieties and the
Mavlyutov duality}

\author{Victor Batyrev}
\address{Fachbereich Mathematik, Universit\"at T\"ubingen, Auf der
Morgenstelle 10, 72076 T\"ubingen, Germany}
\curraddr{}
\email{batyrev@math.uni-tuebingen.de}
\thanks{}

\dedicatory{Dedicated to Yuri Ivanovich Manin on the occasion of his
80-th birthday}

\begin{abstract}
We show that minimal models
of nondegenerated hypersufaces defined by
Laurent polynomials with a $d$-dimensional  Newton polytope
$\Delta$  are  Calabi-Yau varieties $X$ if and only if
the Fine interior of the polytope $\Delta$ consists of a single lattice point.
We give  a combinatorial formula for computing
the stringy Euler number of such Calabi-Yau variety $X$ via the lattice polytope $\Delta$.
This formula allows to test
mirror symmetry in cases when $\Delta$ is not a reflexive polytope.
In particular, we apply this
formula  to pairs of lattice polytopes $(\Delta, \Delta^{\vee})$  that appear in the
Mavlyutov's generalization of the
polar duality for reflexive polytopes. Some examples
of Mavlyutov's dual pairs $(\Delta, \Delta^{\vee})$
show that the stringy Euler numbers of the
corresponding Calabi-Yau varieties
$X$ and $X^{\vee}$  may not satisfy the expected topological mirror symmetry test:
$e_{\rm st}(X) = (-1)^{d-1} e_{\rm st}(X^{\vee})$.
This shows the necessity of an additional condition on Mavlyutov's pairs
$(\Delta, \Delta^\vee)$.
 \end{abstract}


\maketitle

\section{Introduction}

Many examples of pairs of mirror symmetric
Calabi-Yau mani\-folds $X$ and $X^*$
can be obtained using Calabi-Yau hypersurfaces in Gorenstein toric Fano varieties
corresponding to pairs $(\Delta, \Delta^*)$
of $d$-dimensional reflexive lattice polytopes $\Delta$ and $\Delta^*$
that are polar
dual to each other
\cite{Bat94}.

A $d$-dimensional convex
polytope $\Delta \subset \R^d$ is called a {\em lattice polytope} if $\Delta =
{\rm Conv}(\Delta \cap \Z^d)$, i. e., all vertices of $\Delta$ belong to the
lattice $\Z^d \subset \R^d$.  If a $d$-dimensional polytope
$\Delta$ contains the zero $0 \in \Z^d$ in its interior, one defines the
polar polytope $\Delta^* \subset \R^d$
as
\[ \Delta^*:= \{ y \in \R^d \; : \;  \langle x, y \rangle \geq -1 \;\; \forall x \in \Delta\}, \]
where $\langle *, * \rangle \, :\, \R^d \times \R^d \to \R$
is the standard scalar product on $\R^d$. A $d$-dimensional
lattice polytope $\Delta \subset \R^d$
containing $0$ in its interior is called
{\em reflexive} if the polar polytope $\Delta^*$ is also a lattice polytope.
If $\Delta$ is reflexive then $\Delta^*$ is also reflexive and   one has
$(\Delta^*)^*= \Delta$.

For a $d$-dimensional reflexive polytope $\Delta$ one considers
the familiy of Laurent polynomials
\[ f(x) = \sum_{ m \in \Z^d \cap \Delta} a_m x^m  \in
\C[x_1^{\pm 1}, \ldots, x_d^{\pm 1}] \]
with sufficiently general coefficients $a_m \in \C$.  Using the theory of toric
varieties (see e.g. \cite{Ful93, CLS11}),  one can prove
that the affine hypersurface $Z \subset
\T_d:= (\C^*)^d$ defined by $f=0$ is birational to a
$(d-1)$-dimensional Calabi-Yau variety $X$.  In the same way
one obtains another $(d-1)$-dimensional Calabi-Yau variety $X^*$ corresponding
to the polar reflexive polytope $\Delta^*$.

The polar duality $\Delta \leftrightarrow \Delta^*$
between the reflexive polytopes $\Delta$ and
$\Delta^*$ defines a duality between their proper faces
$\Theta  \leftrightarrow \Theta^*$ $(\Theta \prec \Delta$,
$\Theta^* \prec \Delta^*)$ satisfying the condition
$\dim \Theta + \dim \Theta^* =d-1$, where the dual
face $\Theta^* \prec \Delta^*$ is defined as
\[ \Theta^*:= \{ y \in \Delta^*\, :\, \langle x, y \rangle = -1\;\; \forall x \in \Theta \}. \]
There is a simple combinatorial formula for computing the
stringy Euler number of the Calabi-Yau
manifold $X$  \cite[Corollary 7.10]{BD96}:
\begin{align}  \label{e-reflex}
 e_{\rm str}(X)
= \sum_{k=1}^{d-2}  (-1)^{k-1} \sum_{\Theta \preceq \Delta \atop \dim \Theta =k}
v(\Theta) \cdot v\left(\Theta^*\right),
\end{align}
where $v(P):= (\dim P)! Vol(P) \in \Z$ denotes the integral volume
of a lattice polytope $P$.  An alternative proof of the formula
(\ref{e-reflex}) together
with its generalizations for Calabi-Yau complete intersections is
contained in \cite{BS17}.  The formula (\ref{e-reflex}) and  the 
duality $\Theta \leftrightarrow  \Theta^* $ between faces of reflexive polytopes $\Delta$ and
$\Delta^*$ immediately imply the equality
\[   e_{\rm str}(X)  = (-1)^{d-1}  e_{\rm str}(X^*) \]
which is a consequence of a  stronger  topological mirror symmetry test
for the stringy Hodge numbers \cite{BB96}:
\[  h^{p, q}_{\rm str}(X)= h^{d-1 -p,q}_{\rm str} (X^*), \;\;  0 \leq p,q \leq d-1. \]
\medskip

It is important to mention another combinatorial
mirror
construction suggested by Berglung and H\"ubsch \cite{BH\"u93}
and generalized by Krawitz \cite{Kra09}. This mirror construction
considers $(d-1)$-dimensional Calabi-Yau varieties $X$
which are birational to affine hypersurfaces $Z \subset
\T_d$ defined by  Laurent polynomials
\[ f(x) = \sum_{ m \in \Z^d \cap \Delta} a_m x^m  \in
\C[x_1^{\pm 1}, \ldots, x_d^{\pm 1}],  \]
whose  Newton polytope $\Delta  \subset \R^d$ is a lattice simplex, but it is important 
 to stress that this simplex may be not a reflexive simplex. The mirror duality
for the stringy (or orbifold) Hodge numbers of Calabi-Yau varieties obtained by
Berglung-H\"ubsch-Krawitz mirror construction was
proved by Chiodo and Ruan \cite{CR11} and Borisov \cite{Bor13}.

The Batyrev mirror construction \cite{Bat94} and  the
Berglund-H\"ubsch-Krawitz mirror construction \cite{BH\"u93,Kra09}
can be applied to different classes of hypersurfaces in toric varieties, but
they coincide for Calabi-Yau hypersurfaces of Fermat-type.  So
it is natural to expect that there must be a generalization of two
mirror constructions that includes both as special cases
(see \cite{AP15,Bor13,ACG16,Pum11,BH\"u16}). Moreover, it is natural to
expect the existence  of  a  generalization of  combinatorial formula (\ref{e-reflex}) 
for the stringy Euler number $e_{\rm str}(X)$ of
Calabi-Yau varieties $X$ which holds true for projective varieties coming from  a wider
 class of nondegenerate affine hypersurfaces $Z \subset \T_d$ defined
 by Laurent polynomials.

Recall that the stringy Euler number  $e_{\rm str}(X)$ can be
defined for an arbitrary $n$-dimensional normal projective
$\Q$-Gorenstein variety $X$  with at worst
log-terminal singularities using a desingularization $\rho\,:\, Y \to X$
whose exceptional locus is a union of smooth
irreducible divisors $D_1, \ldots, D_s$ with only
normal crossings \cite{Bat98}. For this purpose, one sets
$I:=\{1, \ldots, s\}$, $D_\emptyset := Y$  and for any nonempty
subset $J \subseteq  I$ one defines $D_J := \bigcap_{j \in J}$.
Using the rational coefficients $a_1, \ldots, a_s$ from the
formula
\[ K_Y = \rho^*K_X + \sum_{i=1}^s a_i D_i, \]
one defines the stringy Euler number
\[ e_{\rm str}(X):= \sum_{\emptyset \subseteq J \subseteq I}
e(D_J)\prod_{j \in J} \left( \frac{-a_j}{a_j+1} \right).  \]
Using methods of a  nonarchimedean integration (see e.g. \cite{Bat98,Bat99}),
one can show that $e_{\rm str}(X)$  is independent of the choice of the desingularization $\rho\, :\, Y \to X$ and one has $e_{\rm str}(X) = e_{\rm str}(X')$
if two projective Calabi-Yau varieties $X$ and $X'$ with at worst
canonical singularities are
birational.  More generally,  the stringy Euler number  $e_{\rm str}(X)$
of any minimal projective algebraic variety  $X$ does not depend
on the choice of this  model and coincides with the stringy Euler number of
its canonical model, because all these birational models are $K$-equivalent to each other.  There exist some
versions of the stringy Euler number that are conjectured  to have minimum
exactly on minimal models in a given birational class \cite{BG18}.

We remark that in general the stringy Euler number may not be an integer, and
so far no example of mirror symmetry is known
if the stringy Euler number $e_{\rm str}(X)$ of a Calabi-Yau variety $X$
is not an integer.

In Section 2 we give  a review of results of Ishii \cite{Ish99} on
minimal
models of nondegenerate hypersurfaces and give a
combinatorial criterion that describes
all $d$-dimensional lattice polytopes $\Delta$  such that
minimal models of $\Delta$-nondegenerate hypersurfaces
$Z \subset \T_d$ are Calabi-Yau varieties.  We show
that a $\Delta$-nondegenerate hypersurface $Z\subset \T_d$ is
birational to a Calabi-Yau variety $X$ with at worst $\Q$-factorial
terminal singularities if and only if the Fine interior $\Delta^{FI}$
of its Newton polytope $\Delta$ consists of a single lattice point
(Theorem \ref{fi-cy1}). We remark that
there exist many  $d$-dimensional  lattice polytopes $\Delta$ such that
$\Delta^{FI} = 0 \in \Z^d$ which are not reflexive if $d \geq 3$.

In Section 3 we discuss the generalized combinatorial duality suggested by Mavlyutov
in \cite{Mav11} . Lattice polytopes $\Delta$ that appear in the Mavlyutov duality satisfy
 not only the condition  $\Delta^{FI}=0$, but also
 the additional condition
 $[[ \Delta^*]^*]  = \Delta$,
where $P^*$ denotes the polar polytope of $P$ and $[P]$ denotes the convex hull
of all lattice points in $P$.

 The lattice polytopes $\Delta$ with $\Delta^{FI} = 0$
satisfying
the condition $[[ \Delta^*]^*]  = \Delta$ we call {\em pseudoreflexive}.
A lattice polytope  $\Delta$ with $\Delta^{FI} =0$
may not be a pseudoreflexive, but its
 {\em Mavlyutov dual polytope} $\Delta^\vee:= [\Delta^*]$ and
the lattice polytope $[[ \Delta^*]^*]$ are always
pseudoreflexive. Moreover, if $\Delta^{FI} =0$ then $[[ \Delta^*]^*]$ is
the smallest pseudoreflexive polytope containing $\Delta$.
For this reason we call $d$-dimensional
lattice polytopes $\Delta$  with the only condition $\Delta^{FI} =0$
{\em almost pseudoreflexive}.

If the lattice polytope $\Delta$ is pseudoreflexive, then one
has $((\Delta^\vee)^\vee) =\Delta$. Any  reflexive polytope $\Delta$
is pseudoreflexive, because in this case one has
$\Delta^\vee = [\Delta^*] = \Delta^*$. Therefore  the Mavlyutov
 duality $\Delta \leftrightarrow \Delta^\vee$
is a generalization of the polar duality $\Delta \leftrightarrow \Delta^*$
for reflexive polytopes.

Unfortunately  Mavlyutov dual pseudoreflexive
 polytopes $\Delta$ and
$\Delta^\vee$ are not necessarily
combinatorially dual to each other.
For this reason we can not expect a natural duality between
$k$-dimensional faces of pseudoreflexive polytope
$\Delta$ and $(d-1-k)$-dimensional
faces of its dual pseudoreflexive polytope $\Delta^\vee$. Mavlyutov observed
that  a natural duality
can be obtained if ones restricts attention to some
part of  faces of $\Delta$ \cite{Mav13}.
 A proper $k$-dimensional face $\Theta \prec \Delta$ of a pseudoreflexive
polytope $\Delta$ will be called
 {\em regular} if $\dim [\Theta^*] =d-k-1$,
  where $\Theta^*$ is the dual face of the polar polytope $\Delta^*$.
If $\Theta \prec \Delta$ is a regular
  face of a pseudoreflexive polytope
$\Delta$  then $\Theta^\vee:= [\Theta^*]$ is a regular face of the Mavlyutov dual
pseudoreflexive polytope $\Delta^\vee$
and one has $(\Theta^\vee)^\vee = \Theta$, so that one obtains  a natural duality
between $k$-dimensional regular faces of
$\Delta$ and $(d-k-1)$-dimensional regular faces of
$\Delta^\vee$.  Mavlyutov hoped
that this duality could help to find
a mirror symmetric generalization of the formula
(\ref{e-reflex}) for arbitrary pairs $(\Delta, \Delta^\vee)$
of pseudoreflexive polytopes \cite{Mav13}.

In Section 4 we are interested in a
combinatorial formula for the stringy $E$-function
$E_{\rm str}(X; u,v)$ of a canonical Calabi-Yau model $X$
of a $\Delta$-nondegenerated hypersurface
 for  an arbitrary $d$-dimensional
almost pseudoreflexive polytope
$\Delta$.
Using the results of Danilov and Khovanskii \cite{DKh86}, we  obtain
such a combinatorial formula
for the stringy function   $E_{\rm str} (X; u, 1)$ (Theorem \ref{E-f-u}) and for
the stringy Euler number
$e_{\rm str}(X) :=
E_{\rm str} (X; 1, 1)$ (Theorem \ref{E-f-u1}):
\begin{align} \label{e-fine}
  e_{\rm st}(X) = \sum_{k=0}^{d-1} \sum_{\Theta \subseteq \Delta  \atop \dim \Theta =d-k}
(-1)^{d-1 - k} v(\Theta) \cdot v(\sigma^{\Theta} \cap \Delta^*).
\end{align}
In this formula the polar polytope $\Delta^*$ is in general a rational polytope, the integer $v(\Theta)$
denotes the integral volume of a $(d-k)$-dimensional face
$\Theta \preceq \Delta$ and the rational number 
$v(\sigma^\Theta \cap \Delta^*)$ denotes the integral volume of the $k$-dimensional
rational polytope
$\sigma^\Theta \cap \Delta^*$ contained in the
   $k$-dimensional normal
cone  $\sigma^\Theta$ corresponding to the face $\Theta \preceq \Delta$
in the normal fan of the polytope $\Delta$.
One can easily see that  the formula (\ref{e-reflex}) can be considered as a particular case
of the formula (\ref{e-fine}) if $\Delta $ is a reflexive polytope.

In Section 5 we consider examples of Mavlyutov pairs $(\Delta, \Delta^\vee)$
of pseudoreflexive polytopes obtained from Newton polytopes of  polynomials defining
Calabi-Yau hypersurfaces $X$ of degree $a +d$ in the $d$-dimensional weighted
projective spaces $\P(a,1, \ldots, 1)$ of dimension $d \geq 5$ such that
the weight $a$ does not
divide the degree $a+d$ and $a < d/2$.
These pseudoreflexive
polytopes   $\Delta$  and $\Delta^\vee$  are not reflexive. If $d = ab +1$ for an integer
$b \geq 2$ then $X$ is quasi-smooth and one can apply
the Berglund-H\"ubsch-Krawitz mirror construction. We compute
the stringy Euler numbers of Calabi-Yau hypersurfaces $X$ and their mirrors
$X^\vee$. In particular, we show that
the equality $e_{str}(X) = (-1)^{d-1} e_{str}(X^\vee)$
holds if $d = ab+1$ and in this case one obtains quasi-smooth Calabi-Yau hypersurfaces.  
However, if $d = ab +l$  $(2 \leq l \leq a-1)$, then the
Calabi-Yau hypersurfaces $X \subset \P(a,1, \ldots,1)$ are not quasi-smooth. Using our formulas
for the stringy Euler numbers $e_{\rm str}(X)$ and $e_{\rm str}(X^\vee)$
we show that the equality $e_{\rm str}(X) = (-1)^{d-1}e_{\rm str}(X^\vee)$
can not be satisfied if e.g. $d = ab +2$, where $a, b$ are two distinct odd prime
numbers (Theorem \ref{theo-prime}).

In Section 6 we investigate  the Mavlyutov duality $\Delta \leftrightarrow \Delta^\vee$  together with
an additional condition on singular facets of the pseudoreflexive polytopes $\Delta$ and
$\Delta^\vee$. This condition can be considered as a version  of a quasi-smoothness
condition on Mavlyutov's pairs $(\Delta, \Delta^\vee)$ suggested in some form by Borisov \cite{Bor13}.
For Calabi-Yau varieties
$X$ and $X^\vee$ corresponding to Mavlyutov's pairs $(\Delta, \Delta^\vee)$
satisfying this additional condition
we prove another   generalization of the  formula (\ref{e-reflex})
such that
the equality $e_{\rm str}(X) = (-1)^{d-1}e_{\rm str}(X^\vee)$ holds
(Theorem \ref{theo-cond}).

{\bf Acknowledgements.} It is a pleasure to express  my thanks to
Anvar Mavlyutov, Lev Borisov, Duco van Straten, Harald Skarke, Shihoko Ishii,
Makoto Miura, Karin Schaller and Alexander Kasprzyk
for useful stimulation discussions.

\section{Minimal models of nondegenerate
hypersurfaces}

Let $M \cong \Z^d$ be a free abelian group of rank $d$ and $M_\R = 
M \otimes \R$. Denote by $N_{\R}$ the
dual space ${\rm Hom}(M, \R)$ with the natural pairing
$\langle *, * \rangle \,:\, M_\R \times N_R\ \to \R$.

\begin{definition}
Let $P  = {\rm Conv}(x_1, \ldots, x_k) \subset M_\R$ be a convex polytope 
obtained as the convex hull of a finite subset $\{x_1, \ldots, x_k \}
\subset M_\R$.  Define the piecewise linear function
$${\rm ord}_P \, :\, N_{\R} \to \R $$
as
\[ {\rm ord}_P(y) := \min_{x \in P} \langle x, y \rangle = \min_{i=1}^k  \langle x_i, y \rangle. \]
We associate with $P$ its  {\em normal fan} $\Sigma^P$ which is finite collection of
{\em normal cones} $\sigma^Q$ in the dual space $N_\R$ parametrized by faces $Q \preceq P$. The the cone $\sigma^Q$ is defined as
\[ \sigma^Q := \{ y \in N_\R\; :\; {\rm ord}_P(y) =  \langle x, y \rangle, \;\; \forall x \in Q \}. \]
The zero $0 \in N_\R$ is considered as the normal cone to $P$. One has  
\[ N_\R = \bigcup_{Q \preceq P } \sigma^Q. \]
If  $P \subset M_\R$ is  a $d$-dimensional polytope containing $0 \in M$ in its interior,
then we call
\[ P^*:= \{ y \in N_\R\; :\; {\rm ord}_P(y) \geq -1  \} \]
the {\em polar polytope} of $P$. The polar polytope $P^*$ is the union over
all proper faces $Q \prec P$ of the subsets
\[ P^* \cap \sigma^Q =
\{ y \in \sigma^Q \; :\;  \langle x, y \rangle \geq -1, \;\; \forall x \in Q   \}, \]
i.e.,
 \[ P^* = \bigcup_{Q \prec P} \left( P^* \cap \sigma^Q \right). \]
\end{definition}

\begin{definition}
Let $n \in N$ be a primitive lattice vector and let $l \in \Z$. We consider an affine hyperplane $H_n(l) \subset
M_\R$ defined
by the equation $\langle x, n \rangle = l$. If $m \in M$ then the nonnegative integer
\[  |\langle m, n \rangle - l| \]
is called the {\em integral distance} between  $m$ and  the hyperplane $H_n(l)$.
\end{definition}

\begin{definition}
Let $\Delta \subset M_\R$ be a lattice polytope, i.e., all vertices of $\Delta$ belong to $M$.  Then
${\rm ord}_\Delta$ has integral values on $N$.

For any nonzero
lattice point $n \in N$ one defines the following two half-spaces in $M_\R$:
\begin{align*} \Gamma_0^\Delta(n) &:= \{ x \in M_\R \; :\;  \langle x, n \rangle
\geq {\rm ord}_\Delta(n) \},  \\
 \Gamma_1^\Delta(n) &:= \{ x \in M_\R \; :\;  \langle x, n \rangle
\geq {\rm ord}_\Delta(n) +1 \}.
\end{align*}
\end{definition}

For all $n \in N$ we have obvious  the inclusion
$\Gamma_1^\Delta(n)  \subset  \Gamma_0^\Delta(n)$
and the lattice polytope $\Delta$ can be written as intersection
\[ \Delta = \bigcap_{0 \neq n \in N} \Gamma_0^\Delta(n). \]

\begin{definition}
Let $\Delta$ be a $d$-dimensional lattice polytope.
The {\em Fine interior} of $\Delta$ is defined as
\[ \Delta^{FI}:= \bigcap_{0 \neq n \in N} \Gamma_1^\Delta(n). \]
\end{definition}

\begin{remark}
It is clear that  $\Delta^{FI}$ is a convex subset in the interior of $\Delta$.
We remark that the interior of a $d$-dimensional polytope $\Delta$ is always nonempty, but the
Fine interior of $\Delta$  may be sometimes empty.
I was told that the subset $\Delta^{FI}\subset \Delta$  first has appeared in
the PhD thesis of Jonathan Fine \cite{Fine83}.
\end{remark}

\begin{remark}
Since $\Delta^{FI}$ is defined as an intersection of
countably many half-spaces $ \Gamma_1^\Delta(n)$ it is not immediately clear that
the polyhedral set $\Delta^{FI}$ has only finitely many faces. The latter follows from
the fact that for any proper face $\Theta \prec \Delta$ the semigroup $S_\Theta:= N \cap \sigma^\Theta$
of  all lattice points in the cone $\sigma^\Theta$
is finitely generated (Gordan's lemma).  One can show that $\Delta^{FI}$ can be
obtained as a finite intersection of those half-spaces $\Gamma_1^\Delta(n)$ such
that the lattice vector $n$ appears as a minimal generator of the semigroup
$S_\Theta$ for some face $\Theta \prec \Delta$.
Indeed, if $n', n''  \in \sigma^\Theta$, i.e., if  two
lattice vectors $n',  n'' $ are in the same cone $\sigma^{\Theta}$, and if $x \in \Delta$ is a point in
$\Gamma_1^\Delta(n') \cap \Gamma_1^\Delta(n'')$, then
we have
$$\langle x, n'+ n'' \rangle =   \langle x, n' \rangle + \langle x, n'' \rangle \geq {\rm ord}_\Delta(n') +1 +
{\rm ord}_\Delta(n'') +1 > {\rm ord}_\Delta(n'+ n'') + 1 ,  $$
i. e. , $\Gamma_1^\Delta(n') \cap \Gamma_1^\Delta(n'')$ is contained in $\Gamma_1^\Delta(n'+n'')$
\end{remark}

\begin{remark} \label{FI-lattice}
Let $\Delta \subset M_\R$ be a $d$-dimensional lattice polytope. If
$m \in \Delta$ is an interior lattice point, then $m \in \Delta^{FI}$. Indeed,
if $m \in \Delta$
is an interior lattice point,  then for any lattice point $n \in N$ one has
$\langle m, n \rangle  > {\rm ord}_\Delta(n) $. Since
both numbers
 $\langle m, n \rangle$  and ${\rm ord}_\Delta(n)$ are integers, we obtain
 \[ \langle m, n \rangle \geq {\rm ord}_\Delta(n) + 1, \;\; \forall n \in N,  \]
 i.e., $m$ belongs to $\Delta^{FI}$. This implies the inclusion 
 \[ {\rm Conv}( {Int }(\Delta) \cap M) \subseteq \Delta^{FI}, \]
 i.e., the Fine interior of $\Delta$ contains the convex hull of the set interior lattice points 
 in $\Delta$. 
\end{remark}

We see below that for $2$-dimensional lattice polytopes this inclusion is equality. 
In order to find the Fine  interior of an arbitrary
$2$-dimensional lattice polytope we will
use the following well-known fact:

\begin{prop} \label{tri-2}
Let $\Delta \subset \R^2$ be a lattice triangle such that
$\Delta \cap \Z^2$ consists
of vertices of $\Delta$. Then $\Delta$  is isomorphic
to the standard triangle with vertices $(0,0), (1,0), (0,1)$.
In particular, the integral distance
between a vertex of $\Delta$ and its opposite side of $\Delta$ is always
$1$.
\end{prop}

\begin{prop} \label{FI-2}
If $\Delta$ is a $2$-dimensional lattice polytope, then
$\Delta^{FI}$ is exactly the convex hull of interior lattice
points in $\Delta$.
\end{prop}

\begin{proof}
Let  $\Delta$ be a $2$-dimensional lattice polytope.
If $\Delta$ has no interior lattice points,
then $\Delta$ is isomorphic to either a lattice polytope
in $\R^2$ contained in the strip $0 \leq x_1 \leq 1$,
or to the lattice triangle with vertices
 $(0,0), (2,0), (0,2)$
(see e. g. \cite{Kho97}). In both cases one can
easily check that $\Delta^{FI} = \emptyset$.

If $\Delta$ has exactly one interior lattice point
then $\Delta$ is isomorphic to one of 16 reflexive polygons
and one can check that this interior
lattice point is exactly the Fine interior
of $\Delta$, because, by \ref{tri-2}, this interior lattice point has integral
distance $1$ to its sides.

If $\Delta$ is a $2$-dimensional lattice polytope with at least two interior lattice points then we
denote   $\Delta': = {\rm Conv}(Int(\Delta) \cap M)$. One has $\dim \Delta' \in \{1, 2 \}$.

If $\dim \Delta'=1$ then by \ref{tri-2} the integral distance from the affine line $L$ containing $\Delta'$ and
any lattice vertex of  $\Delta$ outside of this line must be $1$.
This implies that the Fine interior
$\Delta^{FI}$ is contained in $L$. By \ref{FI-lattice},   if $A$ und $B$ are
two vertices of the segment $\Delta'$ then
$A, B \in \Delta^{FI}$. By \ref{tri-2}, there exist a side of $\Delta$ with the integral distance $1$ from $A$ having nonempty intersection with
the line $L$. Therefore, $A$ is vertex of $\Delta^{FI}$. Analogously, $B$ is also a vertex of $\Delta^{FI}$.

Assume now that  $\dim \Delta'=2$, i.e., $\Delta'$ is $k$-gon.  Then
$\Delta'$ is an  intersection of $k$ half-planes $\Gamma_1,
\ldots, \Gamma_k$
whose boundaries are $k$ lines $L_1, \ldots, L_k$ through the $k$ sides
of $\Delta'$. By  \ref{tri-2}, for any $1 \leq i \leq k$
all vertices of $\Delta$  outside the half-plane $\Gamma_i$ must have
integral distance $1$ from $L_i$. Therefore $\Gamma_i$ contains the Fine interior
$\Delta^{FI}$ and $\Delta^{FI} \subseteq \bigcap_{i=1}^k \Gamma_i =
\Delta'$. The opposite inclusion $\Delta' \subseteq \Delta^{FI}$ follows
from \ref{FI-lattice}.
\end{proof}

\begin{remark} \label{3-dim}
The convex hull of all interior lattice points in a lattice polytope $\Delta$
of dimension $d \geq 3$ must not coincide with the Fine interior $\Delta^{FI}$
in general. For example there
exist $3$-dimensional lattice polytopes $\Delta$ without interior lattice
points such that $\Delta^{FI}$ is not empty \cite{Tr08}.  The simplest well-known 
example of such a  situation is the $3$-dimensional lattice simplex corresponding 
to Newton polytope of the Godeaux surface obtained  as a free cyclic group 
of order 5 quotient of  the Fermat surface $w^5 + x^5 + y^5 + z^5 = 0$ 
by the mapping $(w : x : y : z) \to (w: \rho x: \rho^2 y : \rho^3 z)$,  where $\rho$ 
is a fifth root of $1$.
\end{remark}

\begin{definition} \label{supp-fi-def}
Assume that  the lattice polytope $\Delta$ has a nonempty
Fine interior $\Delta^{FI}$.
We define the {\em support of $\Delta^{FI}$} as
\[ {\rm Supp}(\Delta^{FI}):= \{ n \in N\, :\, \langle x, n \rangle ={\rm ord}_\Delta(n) +1
\;\; \mbox{\rm for some $x \in \Delta^{FI}$} \}  \subset N. \]
The convex rational polytope
\[ \Delta^{can} := \bigcap_{n \in  {\rm Supp}(\Delta^{FI})} \Gamma_0^\Delta(n) \]
containing $\Delta$  we call the {\em canonical hull} of $\Delta$.
\end{definition}

\begin{remark}
The support of $\Delta^{FI}$ is a finite subset in the lattice $N$, because it is contained in the union over all
faces $\Theta \prec \Delta$ of all minimal generating
subsets for the semigroups $N \cap \sigma^\Theta$. In particular, ${\rm Supp}(\Delta^{FI})$ always
consists of finitely many primitive nonzero lattice vectors $v_1, \ldots, v_l \in N$ such that
$\sum_{i=1}^l \R_{\geq 0}v_i = N_{\R}$.
 \end{remark}

Let us now identify the lattice $M \cong \Z^d$ with the set of monomials $x^m = x_1^{m_1} \cdots x_d^{m_d}$
in the Laurent polynomial ring $\C[x_1^{\pm 1}, \ldots,  x_d^{\pm 1}] \cong \C[M]$.
For simplicity we will consider only algebraic varieties $X$  over the 
algebraically closed field $\C$.

\begin{definition}
The  \emph{Newton polytope} $\Delta (f)$
of a Laurent polynomial $f(x) = \sum_m a_m x^m
 \in \C[x_1^{\pm 1}, \ldots,  x_d^{\pm 1}]$ is the convex hull of all lattice points
$m \in M$ such that $a_m \neq 0$. For any face $\Theta \subseteq \Delta(f)$ one
defines the $\Theta$-part of the Laurent polynomial $f$ as
$$f_\Theta(x): = \sum_{m \in M \cap \Theta} a_m x^m.$$
A Laurent polynomial $f \in  \C[M]$ with a Newton polytope $\Delta$ is called
\emph{$\Delta$-nondegenerate} or simply \emph{nondegenerate} if for every
face $\Theta \subseteq \Delta$ the zero locus $Z_\Theta:=
\{ x \in \T_d \,:\, f_\Theta(x) =0\}$ of the  $\Theta$-part of $f$ is
 empty or a smooth affine hypersurface in the
$d$-dimensional algebraic
torus $\T_d$.
\end{definition}

The theory of toric varieties allows to  construct a smooth projective algebraic
variety $\widehat {Z}_\Delta$ that contains the affine  $\Delta$-nondegenerate
hypersurface
$Z_\Delta \subset T$
as a Zariski open subset. For this pupose one first considers the closure
$\overline{Z}_\Delta$ of $Z_\Delta$ in the projective toric
variety $\P_\Delta$ associated with the normal fan $\Sigma^\Delta$. Then
one chooses a regular
simplicial
subdivision $\widehat{\Sigma}$ of the fan $\Sigma^\Delta$ and obtains a projective
morphism $\rho\, :\, \P_{\widehat{\Sigma}}  \to \P_\Delta$  from a smooth toric variety
 $\P_{\widehat{\Sigma}}$   to
 $\P_\Delta$  such that, by Bertini theorem, 
 its restriction to the Zariski closure $\widehat {Z}_\Delta$
  of $\Z_\Delta$ in $\P_{\widehat{\Sigma}}$ is a smooth and projective desingularization
of $\overline{Z}_\Delta$.

Now one can apply  the Minimal Model Program of Mori to the smooth
projective hypersurface $\widehat {Z}_\Delta$  (see e.g. \cite{Mat02}).
One can show
that for $\Delta$-nondegenerate hypersurfaces a minimal model
of $\widehat {Z}_\Delta$ can be obtained
via the toric Mori theory due to Miles Reid \cite{Reid83,Wi02,Fu03,FS04} applied
to  pairs $(V, D)$ consisting of  a projective toric variety $V$
and  the Zariski closure $D$ of the nondegenerate hypersurface
$Z_\Delta$ in $V$ \cite{Ish99}.  Therefore, minimal models of nondegenerate
hypersurfaces $Z_\Delta$ can be constructed by combinatorial methods.

Recall the following standard definitions from the
Minimal Model Program \cite{Ko13}.

\begin{definition}
Let $X$ be a normal projective algebraic variety over
$\C$, and let $K_X$ be its canonical class. A birational morphism $\rho \colon Y \to X$
is called  a \emph{log-desingularization} of $X$ if $Y$ is
smooth and the exceptional
locus of $\rho$ consists of smooth irreducible divisors
$D_1, \ldots, D_k$ with simple normal crossings. Assume that
$X$ is $\Q$-Gorenstein, i.e., some integral multiple of
$K_X$ is a Cartier divisor on $X$. We set $I := \{ 1, \ldots, k\}$,
$D_{\emptyset} :=Y$, and for any nonempty subset $J \subseteq I$
we denote by $D_J$ the intersection of divisors $\bigcap_{j \in J} D_j$,
which is either empty or a smooth projective subvariety
in $Y$ of codimension $|J|$. Then the canonical classes
of $X$ and $Y$ are related by the formula
\[ K_Y = \rho^* K_X + \sum_{i=1}^k a_i D_i\text{,} \]
where $a_1, \ldots, a_k$ are rational numbers which are called
\emph{discrepancies}. The singularities of $X$ are said to
be
\[ \begin{array}{lccr}  \emph{log-terminal} &  \mbox{ if} & \mbox{
$a_i >  -1$}   & \forall i \in I; \\
\emph{canonical} &  \mbox{ if} & \mbox{
$a_i \geq 0$}   & \forall i \in I; \\
\emph{terminal} &  \mbox{ if} & \mbox{
$a_i >  0$}   & \forall i \in I.
\end{array}
\]
\end{definition}

It is known that $\Q$-Gorenstein toric varieties always have
at worst log-terminal singularities \cite{Reid83}.

\begin{definition}
A projective normal $\Q$-Gorenstein
algebraic variety $Y$ is called {\em canonical model}
of $X$ if $Y$ is birationally equivalent to $X$,
$Y$ has at worst canonical singularities and the linear
system $|mK_Y|$ is base point free for sufficiently
large integer $m \in \N$.
\end{definition}

\begin{definition}
A projective algebraic variety $Y$ is called {\em minimal model}
of $X$ if $Y$ is birationally equivalent to $X$,
$Y$ has at worst terminal $\Q$-factorial  singularities, the
canonical class $K_Y$ is numerically effective,
 and the linear system $|mK_Y|$ is base point free for sufficiently
large integer $m \in \N$.
\end{definition}

 The main result of the Minimal Model Program for nondegenerate hypersurfaces in
 toric varieties in \cite{Ish99}  can be reformulated using combinatorial interpretations
 of Zariski decompositions
 of effective divisors on toric varieties \cite{OP91} (see also \cite[Appendix A]{HKP06})
 and their applications  to log minimal models of polarized
 pairs \cite{BH14}

\begin{definition}
Let $\P_{\Sigma}$ be a $d$-dimensional
projective toric variety defined by a fan $\Sigma$ whose  $1$-dimensional cones $\sigma_i = \R_{\geq 0} v_i \in
\Sigma(1)$ are generated by primitive lattice vectors
$v_1, \ldots, v_s \in N$. Denote by $V_i$ $( 1 \leq i \leq s)$ torus invariant
divisors on $\P$ corresponding to $v_i$.
Let $D= \sum_{i=1}^s a_iV_i$ be an arbitrary torus invariant $\Q$-divisor on
$\P$ such that the rational polytope
\[ \Delta_D := \{  x \in M_\R\; :\; \langle x, v_i \rangle \geq - a_i, \; 1 \leq i \leq s \} \]
is not empty.
Then the  rational numbers
\[ {\rm ord}_{\Delta_D}(v_i):= \min_{ x \in \Delta_D}  \langle x, v_i \rangle,
\;\;  1 \leq i \leq r,  \]
satisfy the inequalities
\[ {\rm ord}_{\Delta_D}(v_i) + a_i \geq 0, \;\;  1 \leq i \leq r . \]
Without loss of generality we can assume that
the equality ${\rm ord}_{\Delta_D}(v_i) + a_i  = 0$ holds
if and only if  $1 \leq i  \leq  r$ $(r \leq s)$.
We define the {\em support of the polytope $\Delta_D$} as
\[ {\rm Supp}(\Delta_D) := \{ v_1,  \ldots, v_r \}  \]
and we write the divisor $D = \sum_{i=1}^l a_iV_i $ as the sum
\[ D = P + N, \;\; P = \sum_{i=1}^s
( - {\rm ord}_{\Delta_D}(v_i) ) V_i, \;\;  N:=  \sum_{i=r+1}^s
( {\rm ord}_{\Delta_D}(v_i) +a_i ) V_i, \]
where $ {\rm ord}_{\Delta_D}(v_i) +a_i > 0$ for all
$ r+ 1 \leq i \leq s$. Then there exists a $d$-dimensional
$\Q$-factorial projective toric variety $\P'$ defined by a simplicial fan
$\Sigma'$ whose
$1$-dimensional cones are generated by the lattice vectors
$v_i \in {\rm Supp}(\Delta_D)$ together with a birational toric
morphism $\varphi\, :\, \P \to \P'$ that contracts the divisors
$V_{r+1}, \ldots, V_s$ such the nef divisor $P$ on $\P$ is
the pull back of the nef divisor
$$P':=   \sum_{ v_i \in {\rm Supp}(\Delta_D)}
(-{\rm ord}_{\Delta_D}(v_i))  V_i'$$
on $\P'$.  The decomposition
$D = P + N$  together with the nef divisor $P'$ on the
 $\Q$-factorial projective toric variety $\P'$
we call {\em toric Zariski decomposition of $D$}.
\end{definition}

\begin{theorem} \label{MMP-nondeg}
A $\Delta$-nondegenerate hypersuface $Z_\Delta \subset \T_d$ has a minimal model
if and only if the Fine interior
$\Delta^{FI}$ is not empty. In the latter case, a canonical
model of the nondegenerate
hypersurface $Z_\Delta$ is
its  closure $X$ in the toric variety $\P_{\Delta^{can}}$
associated to the canonical hull
$\Delta^{can}$ of the lattice polytope $\Delta$. The birational isomorphism
between $\overline{Z}_\Delta$ and $X$ is
induced by the birational isomorphism of toric varieties
$\alpha\,:\,
\P_\Delta  \dashrightarrow \P_{\Delta^{can}}$, it  can be included in a  diagram
\[
\xymatrix{   & \P_{\widehat{\Sigma}}\ar[ld]_{\rho_1} \ar[rd]^{\rho_2} & \\
\P_{\Delta}  \ar@{-->}[rr]^\alpha &  & \P_{\Delta^{can}}}
\]
where $\widehat{\Sigma}$ denotes the common regular
simplicial subdivision of the normal
fans $\Sigma^\Delta$ and $\Sigma^{\Delta^{can}}$. In particular,
one obtains two  birational
morphisms $\rho_1\, :\,  \widehat{Z}_\Delta \to \overline{Z}_\Delta$,
$\rho_2\, :\,
\widehat{Z}_\Delta \to X$ in  the diagram
\[
\xymatrix{   & \widehat{Z}_\Delta\ar[ld]_{\rho_1} \ar[rd]^{\rho_2} & \\
\overline{Z}_\Delta \ar@{-->}[rr]^\alpha &  & X}
\]
where $\widehat{Z}_\Delta$
denotes the Zariski closure of $Z_\Delta$ in  $\P_{\widehat{\Sigma}}$.
\end{theorem}

\begin{proof}
Let $L$ be the ample Cartier divisor on the $d$-dimensional
toric variety $\P_\Delta$ corresponding
to a $d$-dimensional lattice polytope $\Delta$.
We apply the toric Zariski decomposition to the
adjoint divisor $D:= \rho^*L + K_{\P_{\widehat{\Sigma}}}$  for
some toric desingularization $\rho\, : \, \P_{\widehat{\Sigma}} \to \P_\Delta$ defined by
a fan $\widehat{\Sigma}$ which is a regular simplicial subdivision of
the normal fan $\Sigma^\Delta$.
Let $\{v_1, \ldots, v_s\}$ be the set of primitive lattice vectors in $N$
generating $1$-dimensional cones in $\widehat{\Sigma}$.

Since one has $K_{\P_{\widehat{\Sigma}}} = -\sum_{i=1}^s V_i$ and
$\rho^*L = \sum_{i=1}^s (-{\rm ord}_\Delta(v_i)) V_i$,   we
obtain that the rational polytope $\Delta_D$ corresponding to the adjoint
divisor on  $\P_{\widehat{\Sigma}}$
$$D= \rho^*L + K_{\P_{\widehat{\Sigma}}} =
\sum_{i=1}^s ( -{\rm ord}_\Delta(v_i) -1) V_i$$
is exactly
the Fine interior $\Delta^{FI}$ of $\Delta$.

We can
assume that  ${\rm Supp}(\Delta^{FI}) =
\{v_1, \ldots, v_r \} $ $(r \leq s)$  and
the first $l$ lattice vectors $v_1, \ldots, v_l$ $(l \leq r)$
form   the set of generators
of $1$-dimensional cones
$\R v_i$ $(1 \leq i \leq l)$ in the normal fan
$\Sigma^{\Delta^{can}}$ so that one has
\[ \Delta^{can} = \bigcap_{i=1 }^r \Gamma_0^{\Delta}(v_i) =
\bigcap_{i=1}^l  \Gamma_0^{\Delta}(v_i) \]
and
\[ {\rm ord}_\Delta(v_i) + 1 = {\rm ord}_{\Delta^{FI}}(v_i), \;\; \forall i =1, \ldots, r. \]

The toric Zariski decomposition of $D= \rho^*L + K_{\P_{\widehat{\Sigma}}} =
\sum_{i=1}^s ( -{\rm ord}_\Delta(v_i) -1) V_i$ is the sum $P + N$ where
\[ P = \sum_{i=1}^s ( -{\rm ord}_{\Delta^{FI}}(v_i) ) V_i, \]
\[ N = \sum_{i=r+1}^s ( {\rm ord}_{\Delta^{FI}}(v_i) - {\rm ord}_\Delta(v_i) - 1 ) V_i \]
where $( {\rm ord}_{\Delta^{FI}}(v_i) - {\rm ord}_\Delta(v_i) - 1 ) >0$ for all $i > r$.
Moreover, there exists a projective $\Q$-factorial toric variety $\P'$ such that
$v_1, \ldots, v_r$ is the set of primitive lattice generators of $1$-dimensional
cones in the fan $\Sigma'$ defining the toric variety $\P'$ and
$$ \sum_{i=1}^r ( -{\rm ord}_{\Delta^{FI}}(v_i) ) V_i' $$
is a nef $\Q$-Cartier divisor on $\P'$

Therefore the canonical divisor  $K_{\P_{\Delta^{can}}}$ equals
$- \sum_{i =1}^l V_i$ where $V_1, \ldots, V_l$ the set of torus invariant divisors
on $\P_{\Delta^{can}}$. Let $X$ be the Zariski closure
of the affine $\Delta$-nondegenerated
hypersurface $Z_\Delta$ in $\P_{\Delta^{can}}$. Then $X$ is
linearly equivalent to a linear combination $\sum_{i =1}^l b_iV_i$, where
$b_i = - {\rm ord}_\Delta(v_i)$ $(1 \leq i \leq l)$. So we obtain
\[  K_{\P_{\Delta^{can}}} + X \sim \sum_{i =1}^l (- {\rm ord}_\Delta(v_i) -1) V_i =
\sum_{i =1}^l (- {\rm ord}_{\Delta^{FI}}(v_i) -1) V_i. \]
On the other hand, we have
\[  \Delta^{FI} =  \bigcap_{n \in {\rm Supp}(\Delta^{FI}) }
\Gamma_1^{\Delta}(n)  = \bigcap_{i=1}^l  \Gamma_1^{\Delta}(v_i).
\]
So $ K_{\P_{\Delta^{can}}} + X$  is a semiample $\Q$-Cartier divisor
on the projective
toric variety
$\P_{\Delta^{can}}$ corresponding
to the rational convex polytope $\Delta^{FI}$.

For nondegenerate hypersufaces one can apply the
adjunction  and obtain that the canonical
class $K_X$ is the restriction to $X$ of the semiample $\Q$-Cartier divisor
 $K_{\P_{\Delta^{can}}} + X$.
The log-discrepancies of the toric pair $(\P_{\Delta^{can}}, X)$
are equal to the discrepancies of $X$, because of inversion of the
anjunction for non-degenerate hypersurfaces \cite{Amb03}.
\end{proof}

\begin{cor} \label{MMP-nondeg1}
For  the above birational morphism
$\rho_2\, :\, \P_{\widehat{\Sigma}} \to \P_{\Delta^{can}}$
one has
\[ K_{\widehat {\P}_{\Sigma} } + \widehat{Z}_\Delta =
\rho_2^*\left(  K_{\P_{\Delta^{can}}} + X \right) +
\sum_{i=l+1}^s a_i V_i \]
and
\[ K_{ \widehat{Z}_\Delta} = \rho_2^*  K_X  +
\sum_{i=l+1}^s a_i D_i, \;\; D_i:= V_i \cap  \widehat{Z}_\Delta, \]
where $V_i$ denotes the torus invariant divisor on  $\P_{\widehat{\Sigma}}$ corresponding
to the lattice point $v_i \in N$ and
\[  a_i= - {\rm ord}_{\Delta}(v_i)  +
{\rm ord}_{\Delta^{FI}}(v_i)  -1 \geq 0. \]
\end{cor}

By \ref{FI-lattice}, one immediately obtains

\begin{cor}
It a $d$-dimensional lattice polytope $\Delta$ contains an interior lattice point, then
a $\Delta$-nondegenerated affine hypersurface $Z_\Delta \subset \T_d$ has a minimal
model.
\end{cor}

\begin{example} As we have already mentioned in \ref{3-dim} there
exist $3$-dimensional lattice polytopes $\Delta$ without interior lattice
points such that $\Delta^{FI}$ is not empty.
One of such examples
is the $3$-dimensional lattice simplex $\Delta$ such that $\Delta$-nondegenerated
hypersurface is birational to the  Godeaux surface obtained as a quotient
of the Fermat quintic $z_0^5 + z_1^5 + z_2^5 + z_3^5 =0$ in $\P^3$ by the action
of a cyclic group of order 5 $$(z_0 : z_1 : z_2 : z_3)  \mapsto
(z_0:\rho z_1: \rho^2z_2: \rho ^3z_3)$$  where $\rho$ is a
$5$-th root of unity.
\end{example}

We apply the above general results to $\Delta$-nondegenerate
hypersurfaces whose minimal models are
Calabi-Yau varieties. It is known that the number of
interior lattice points in $\Delta$ equals
the geometric genus of the $\Delta$-nondegenerate hypersurface \cite{Kho78}.
Therefore, if a nondegenerate
hypersurface $Z_\Delta$  is birational to a Calabi-Yau variety,
then $\Delta$ must contain exactly
one interior lattice point. However, this condition for $\Delta$ is not sufficient.

\begin{example}
In \cite{CG11}  Corti and Golyshev gave $9$ examples of  $3$-dimensional
lattice simplices $\Delta$ with a single interior lattice point $0$
such that the corresponding nondegenerate hypersurfaces $Z_\Delta$
are not birational to a $K3$-surface.  For example  they consider
hypersurfaces of degree $20$ in the weighted projective space $\P(1, 5,6,8)$.
The corresponding $3$-dimensional lattice simplex $\Delta$
is the convex hull
of the lattice points $(1,0,0), (0,1,0), (0,0,1), (-5,-6,-8)$.  The Fine interior $\Delta^{FI}$
of $\Delta$ is a $1$-dimensional polytope on the ray
generated by the lattice vector $(-1,-1,-2)$.
\end{example}

\begin{theorem} \label{fi-cy}
A canonical model of  $\Delta$-nondegenerate affine hypersuface $Z_\Delta
\subset \T_d$ is birational to a
Calabi-Yau variety $X$ with at worst Gorenstein canonical  singularities
if and only if the Fine interior
$\Delta^{FI}$ of the lattice polytope $\Delta$ consists of a single lattice point.
If $\Delta^{FI} =0$  then  $X$ can be obtained
as a Zariski closure of $Z_\Delta$  in
the toric  $\Q$-Fano variety   $\P_{\Delta^{can}}$ so that
$X$ is an anticanonical divisor on $\P_{\Delta^{can}}$.  There exists
an embedded  desingularization $\rho_2\, :\, \widehat{Z}_\Delta \to X$ and
\[ K_{ \widehat{Z}_\Delta} = \rho_2^*  K_X  +
\sum_{i=l+1}^s  a_i D_i,  \]
where the discrepancy $a_i$ of the exceptional divisor
$D_i:= V_i \cap  \widehat{Z}_\Delta$
on $\widehat{Z}_\Delta$ can be computed by the formula
\[  a_i= -   {\rm ord}_{\Delta}(v_i)  -1 \geq 0. \]
\end{theorem}

\begin{proof}
First of all we remark that this statement has been partially proved
in \cite[Prop. 2.2.]{ACG16}, but the application of  Mori theory
for nondegenerate hypersurfaces (see
\ref{MMP-nondeg}  and \ref{MMP-nondeg1}) imply  stronger
statements. The above formula for the discrepancies $a_i$ is not new
and  it has appeared
already in \cite{CG11} for resolutions
of canonical singularities of Calabi-Yau hypersurfaces $X$ in weighted projective
spaces.  In general case,  one
can make a direct  computation of $a_i$ using
the global nowhere vanishing differential $(d-1)$-form
$\omega$ obtained as the Poincar\'e residue $\omega =
Res \, \Omega$ of the rational differential form \cite{Bat93}:
\[ \Omega:=  \frac{1}{f} \frac{d x_1}{x_1} \wedge \cdots \wedge \frac{d x_d}{x_d}. \]
Since $\Delta$ is the Newton polytope of the Laurent polynomial $f$,
the order of zero of $\omega$ along the exceptional divisor $E_i$ corresponding to
the lattice point $v_i \in N$ equals
 $-{\rm ord}_{\Delta}(v_i)  -1$.
\end{proof}

\begin{remark}
We note that a $d$-dimensional lattice polytope $\Delta$ with $\Delta^{FI} =0$ is
reflexive if and only if  $\Delta = \Delta^{can}$.
\end{remark}

If $\Delta$ is reflexive then the canonical
singularities of the projective Calabi-Yau hypersurface $\overline{Z}_\Delta \subset
\P_{\Delta}$ have a MPCP (maximal projective crepant partial) resolution
obtained from a simplicial fan $\widehat{\Sigma}$ whose
generators of $1$-dimensional cones
are lattice points on the boundary of the polar reflexive polytope $\Delta^*$
\cite{Bat94}. This fact can be generalized to an arbitrary $d$-dimensional lattice
polytope $\Delta \subset M_\R$ such that $\Delta^{FI} =0$.  For this we need
the following statement:

\begin{prop} \label{supp}
Let $\Delta$ be a $d$-dimensional polytope with $\Delta^{FI} =0$. Then
one has
$${\rm Supp}(\Delta^{FI}) =\{ \Delta^* \cap N \} \setminus \{ 0\}, $$ where
$\Delta^*$ is the polar polytope.
\end{prop}

\begin{proof}
By Definition \ref{supp-fi-def}, a lattice point $n \in N$
belongs to the support of the
Fine interior $\Delta^{FI} =0$ if and only if
${\rm ord}_\Delta(n) = -1$. The polar polytope $\Delta^* \subset N_\R$ is defined by
the condition ${\rm ord}_\Delta(x) \geq -1$. Therefore, we obtain
${\rm Supp}(\Delta^{FI}) \subset \Delta^*$.
Since $0$ is an interior
lattice point of $\Delta$ one has  $0 > {\rm ord}_\Delta(n)  \in \Z$ for any nonzero
lattice vector $n \in N$. In particular, one has
${\rm ord}_\Delta(n) = -1$ for any nonzero lattice point $n \in \Delta^*$, i.e.,
$\{ \Delta^* \cap N \} \setminus \{ 0\} \subset {\rm Supp}(\Delta^{FI})$.
\end{proof}

\begin{theorem} \label{fi-cy1}
A minimal model of a $\Delta$-nondegenerate affine hypersuface $Z_\Delta \subset
\T_d$ is birational to a
Calabi-Yau variety $X'$ with at worst $\Q$-factorial Gorenstein
terminal singularities
if and only if the Fine interior of $\Delta$ is $0$.
\end{theorem}

\begin{proof}
By Theorem \ref{fi-cy}, it remains to explain how to construct a
maximal projective  crepant partial resolution $\rho' \; :\; X' \to X$.
We consider  the finite set ${\rm Supp}(\Delta^{FI}) = \{v_1, \ldots, v_r\} \subset N$  consisting of
all nonzero lattice points in  $\Delta^* \subset N_\R$.
(see \ref{supp}).
Denote by $\Sigma$  the fan of cones over all faces of the lattice polytope $[ \Delta^*]$
obtained as convex hull of all lattice points in ${\rm Supp}(\Delta^{FI}) $. The fan $\Sigma$ admits a maximal simplicial projective subivision
$\Sigma'$ which consits of simplicial cones whose generators are nonzero lattice vectors
in $\Delta^*$. Thus we obtain a projective crepant
toric morphism $\rho'\; :\; \P_{\Sigma'} \to \P_\Sigma$. Since $\Sigma$ is the normal
fan to the polytope $\Delta^{can} = [\Delta^*]^*$, the morphism $\rho'$
induces  a projective crepant morphism of Calabi-Yau varieties $ \rho' \; :\; X' \to X$, where $X'$
is the Zariski closure of $Z_\Delta$ in $\P_{\Sigma'}$.  Since the toric singularities
of $\P_{\Sigma'}$ are $\Q$-factorial and terminal, the same is true for
the singularities of $X'$.
\end{proof}

\section{The Mavlyutov  duality}

In \cite{Mav11} Mavlyutov
has proposed a generalization the Batyrev-Borisov duality \cite{BB97}. In particular,
his generalization includes the polar  duality for reflexive
polytopes \cite{Bat94}.
We reformulate the ideas of Mavlyutov about Calabi-Yau hypersurfaces
in toric varieties
in some equivalent  more convenient form.

For simplicity we denote  by  $[P]$ the convex hull
${\rm Conv}(P \cap \Z^d)$ for any subset $P \subset \R^d$.
As above, we denote by $P^*$ the polar set of $P$ if $0$ is an interior lattice point of $P$.

Let $\Delta  \subset M_\R$ be $d$-dimensional lattice polytope
such that the Fine interior of $\Delta$
is  zero, i.e.,
$\Delta^{FI} = 0 \in M$.
By \ref{supp},   the support of the Fine interior ${\rm Supp}(\Delta^{FI})$ is equal to
the set of nonzero lattice points in the polar polytope
$\Delta^* \subset N_{\R}$ and
the zero lattice point $0 \in N$
is the single interior lattice point of $[\Delta^*]$. Therefore,  the inclusion
$[\Delta^*] \subseteq \Delta^*$ implies  the inclusions
\[ \Delta = (\Delta^*)^* \subseteq  [\Delta^*]^* = \Delta^{can} \]
and
\[ \Delta  \subseteq  [ [ \Delta^*]^*] = [\Delta^{can}], \]
because $\Delta$ is a lattice polytope.

\begin{definition}
We call a  $d$-dimensional   lattice polytope
$\Delta \subset M_\R$  with $\Delta^{FI} = 0$  {\em pseudoreflexive} if one has
the equality
\[ \Delta = [ [ \Delta^*]^*]. \]
\end{definition}

\begin{remark}
The above defintion of pseudoreflexive polytopes has been discovered by Mavlyutov
in 2004 (see \cite[Remark 4.7]{Mav05}). In the paper \cite{Mav11} Mavlyutov
called these polytopes {\em $\Z$-reflexive} (or {\em integrally reflexive}).
Independently, this definition has been discovered by Kreuzer
\cite[Definition 3.11]{Kr08} who called such polytopes $IPC$-closed.
\end{remark}

\begin{remark}
Every reflexive polytope $\Delta$ is pseudoreflexive,
because for reflexive polytopes $\Delta$
 we have $\Delta = [\Delta]$ and  $\Delta^* = [\Delta^*]$.
 The converse  is not true
if $\dim \Delta \geq 5$.
For instance the convex hull ${\rm Conv}(e_0, e_1, \ldots, e_5)$
of the standard basis $e_1, \ldots, e_5$ in  $\Z^5$ and the lattice vector $e_0 = -e_1 -e_2-e_3 -e_4 -2e_5$
is a $5$-dimensional
pseudoreflexive simplex which is not reflexive.
\end{remark}

There exist a close connection between lattice polytopes $\Delta$ with $\Delta^{FI} =0$ and pseudoreflexive
polytopes:

\begin{prop} \label{FI0-equiv}
Let $\Delta \subset M_\R$ a $d$-dimensional lattice polytope. Then the following
conditions are equivalent:

 {\rm (i)} $\Delta^{FI} = 0$;

{\rm (ii)}  the polytopes $\Delta$ and  $[\Delta^*]$ contain $0$ in their interior;

{\rm (iii)}  $\Delta$ contains $0$ in its interior and $\Delta$ is contained in a pseudoreflexive polytope.
\end{prop}

\begin{proof}  ${\rm (i)} \Rightarrow {\rm (ii)}$. Assume that $\Delta^{FI} = 0$. Then $0$ is an interior
lattice point of $\Delta$ and  the support
of the Fine interior ${\rm Supp}(\Delta^{FI})$ is exactly the set of nonzero lattice points in
the polar polytope $\Delta^*$. Moreover, one has
\[ 0 =
\{ x \in M_\R\, : \, \langle x, v \rangle \geq 0  \;\; \forall v \in {\rm Supp}(\Delta^{FI}) \}.\]
Hence,  $[\Delta^*]$ also contains $0$ in its interior.

${\rm (ii)} \Rightarrow {\rm (i)}$. If $\Delta$ contains $0$ in its interior, then $0 \in \Delta^{FI}$. For
any nonzero lattice point $v \in \Delta^*$ the minimum of $\langle *, v \rangle$ on $\Delta$ equals
$-1$. If $[\Delta^*]$ contains $0$ in its interior, then there exists lattice points $v_1, \ldots, v_l \in
[\Delta^*]$ generating $M_\R$ such that for some positive numbers $\lambda_i $ $(1 \leq i \leq l )$ one has
\[ \lambda_1 v_1 + \cdots + \lambda_l v_l =0. \]
On the other hand, $\Delta^{FI}$ is contained in the intersection of the half-spaces
 $\langle x, v_i \rangle \geq 0$
$1 \leq i \leq l$. Therefore, one has $\Delta^{FI} =0$.

${\rm (iii)} \Rightarrow {\rm (ii)}$. Assume that $\Delta$ is contained in a pseudoreflexive lattice polytope
$\widetilde{\Delta}$. Then we obtain the inclusions $\widetilde{\Delta}^* \subseteq \Delta^*$ and
$[\widetilde{\Delta}^*] \subseteq [\Delta^*]$ . Since
$\widetilde{\Delta}$ is pseudoreflexive, its Fine interior is zero
and it follows from
${\rm (i)} \Rightarrow {\rm (ii)}$ that $[\widetilde{\Delta}^* ]$
contains $0$ in its interior. Therefore,
the lattice polytope $[\Delta^*]$ also contains $0$ in its interior.

 ${\rm (i} \Rightarrow {\rm (iii)}$. Assume that $\Delta^{FI} =0$.
Then we obtain the inclusion
 $\Delta \subseteq [[ \Delta^*]^*]$. It is sufficient  to show that
$[[ \Delta^*]^*]$ is pseudoreflexive.
 The latter follows from the equality $[[[\Delta^*]^*]^*] = [\Delta^*]$. Indeed,  the inclusion
$\Delta \subseteq [[\Delta^*]^*]$ implies the inclusions $[[\Delta^*]^*]^* \subseteq
\Delta^*$ and $[[[\Delta^*]^*]^*] \subseteq
[\Delta^*]$. On the other hand, the Fine interior of $[\Delta^*]$ is also zero, because
$\Delta$ contains $0$ in its interior. This implies the opposite inclusion
$[\Delta^*] \subseteq [[[\Delta^*]^*]^*]$.
\end{proof}

\begin{cor} \label{ps-almost}
Let $\Delta \subset M_\R$ be a $d$-dimensional lattice polytope with
$\Delta^{FI} =0$. Then the following statements hold.

 {\rm (i)} The lattice polytopes
$[\Delta^*]$ and $[[\Delta^*]^*]$ are pseudoreflexive;

 {\rm (ii)} $[[\Delta^*]^*]$ is the smallest pseudoreflexive
polytope containing $\Delta$.
\end{cor}

\begin{proof}  The statement {\rm (i)} follows from the equality
$[[[\Delta^*]^*]^*] = [\Delta^*]$
in the proof of \ref{FI0-equiv}. If $\widetilde{\Delta}$ is a
pseudoreflexive polytope containing
$\Delta$, then the inclusion $\Delta \subseteq \widetilde{\Delta}$
implies the sequence of inclusions $\widetilde{\Delta}^* \subseteq {\Delta}^*$,
$[\widetilde{\Delta}^*] \subseteq [{\Delta}^*]$,  $[\Delta^*]^*
\subseteq [\widetilde{\Delta}^*]^*$, $[[\Delta^*]^*] \subseteq
[[\widetilde{\Delta}^*]^*]= \widetilde{\Delta}$. This implies  {\rm (ii)}.
\end{proof}

The statements in \ref{FI0-equiv} and \ref{ps-almost} motivate
another names for lattice polytopes $\Delta$ with $\Delta^{FI} =0$:

\begin{definition}
A $d$-dimensional lattice polytope
is called {\em almost pseudoreflexive} if $\Delta^{FI} =0$. If $\Delta$ is
almost pseudoreflexive then
we call the lattice polytope $[[\Delta^*]^*]$  the
{\em pseudoreflexive closure} of $\Delta$ and the lattice polytope
$[\Delta^*]$ the {\em pseudoreflexive dual} of $\Delta$.
The polytope $\Delta$ is called
{\em almost reflexive} if its pseudoreflexive closure
$[[\Delta^*]^*]$ (or, equivalently, its pseudoreflexive dual $[\Delta^*]$)
is reflexive.  In this case, we will call the lattice polytope $[[\Delta^*]^*] =
[\Delta^*]^* = \Delta^{can}$
also the {\em canonical
reflexive closure} of $\Delta$.
\end{definition}

\begin{example}
The $3$-dimensional lattice polytope $\Delta$ obtained as the convex hull of
$(1,0,0), (0,1,0), (0,0,1), (-1,-1,-2) \in \Z^3$ is a  $3$-dimensional
almost reflexive simplex which is not reflexive.
The canonical reflexive closure $[[\Delta^*]^*]
= [\Delta^*]^*$ of  $\Delta$ is a reflexive lattice polytope
obtained from $\Delta$ by adding
one more vertex $(0,0,-1)$.
\end{example}

\begin{remark}
If $\Delta$ is pseudoreflexive, then $\Delta^\vee:= [\Delta^*]$
is also pseudoreflexive.
In particular, one obtains a natural
duality $\Delta \leftrightarrow  \Delta^\vee$
for pseudoreflexive polytopes that generalizes the polar duality
for reflexive polytopes. This duality was suggested by  Mavlyutov in
\cite{Mav11} for unifying different combinatorial mirror constructions.
\end{remark}

\begin{remark}
Unfortunately almost pseudoreflexive polytopes $\Delta$ do not have
a natural duality, although they appear in the
Berglung-H\"ubsch-Krawitz mirror construction. Nevertheless, the pseudoreflexive
duals $[\Delta^*]$ of almost pseudoreflexive polytopes $\Delta$ allow
to connect the Mavlyutov
duality with the Berglung-H\"ubsch-Krawitz mirror construction.
For instance, it may happen that two different almost pseudoreflexive polytopes
$\Delta_1 \neq  \Delta_2$  have the same pseudoreflexive duals, i.e.,
$[\Delta_1^*] = [\Delta_2^*]$. This equality is the key observation for   the birationality of
BHK-mirrors investigated  in \cite{Ke13,Cla14,Sh14}.
\end{remark}

\begin{definition}
Let $\Theta$ be a $k$-dimensional face of a $d$-dimensional
pseudoreflexive polytope
$\Delta \subset M_\R$.  We call $\Theta$  {\em ordinary} if the following
equality holds:
\[  \left( \bigcap_{l \in \Z_{\geq 0} } l \Theta \right) \cap M = \R_{\geq 0} \Theta \cap
M, \]
in other words, if all lattice points in the $(k+1)$-dimensional
cone $\sigma_\Theta =  \R_{\geq 0} \Theta$ over the face
$\Theta \prec \Delta$ are contained in the multiples $l\Theta$ $(l \in \Z_{\geq 0})$.
\end{definition}

\begin{prop} \label{face1}
Let $\Theta \prec \Delta$ be a $k$-dimensional face of a lattice
polytope $\Delta$ with $\Delta^{FI} =0$. Assume that   $[\Theta^*]$ is nonempty.  Then
$\Theta$ is ordinary and $[\Theta^*]$ is
a face of dimension $\leq d-1-k$ of the pseudoreflexive polytope $[\Delta^*]$.
\end{prop}

\begin{proof}
Let $x \in \Theta$ be a point in the relative interior of $\Theta$. The minimum
of the linear function $\langle x, * \rangle$  on $\Delta^*$ equals $-1$ and it is
attained exactly on the polar face
$\Theta^* \prec \Delta^*$ of the rational polar polytope $\Delta^*$.

The minimum $\mu_x$ of $\langle x, * \rangle$ on the lattice
polytope $[\Delta^*]$ is
attained on some lattice face $F \prec [\Delta^*]$ of the lattice polytope
$[\Delta^*]$  such that  $F:= \{ y \in [\Delta^*]\; : \;
 \langle x, y \rangle = \mu_x\}$. The minimum
 $\mu_x$ must be at least  $-1$, because the polytope
$ [\Delta^*]$ is contained in $\Delta^*$.  If $[\Theta^*]$
is not empty then the linear function
$\langle x, * \rangle$ has the constant value $-1$ on $[\Theta^*]$. This implies
that
the minimum $\mu_x$ must be $-1$. Therefore $F$ must be
contained in $\Theta^*$ and $F = [F] \subseteq [\Theta^*]$.
Since $[\Theta^*] \subseteq \{ y \in [\Delta^*]\; : \;
 \langle x, y \rangle = \mu_x=-1\} = F$,  we conclude $F= [\Theta^*]$.
\end{proof}

The next statement is a slight generalization of the results of Skarke in \cite{Sk96}.

\begin{theorem} \cite{Mav13}
Let  $\Theta$ is a face of dimension $k \leq 3$ of a $d$-dimensional
pseudoreflexive polytope $\Delta$. Then $\Theta$ is
ordinary. In particular, any pseudoreflexive lattice polytope $\Delta$ of dimension
$\leq 4$ is reflexive.
\end{theorem}

\begin{proof} Consider the $(k+1)$-dimensional subspace $L:= \R \Theta$ generated by
$\Theta$. Then $\Theta$ is contained in the $k$-dimensional affine hyperplane $H_\Theta$
in $L$. It is enough to show that the integral distance between $H_\Theta$ and $0$ equals $1$.
Assume that this distance is larger than $1$. Consider the pyramid
$\Pi_\Theta := {\rm Conv}(\Theta, 0)$. By lemma of Skarke \cite[Lemma 1]{Sk96},
there exists an interior  lattice point $u_0 \in M$
in $2\Pi_\Theta$ which is not contained in $\Pi_\Theta$. Therefore, the lattice point $u_0$ is an
interior lattice point in the polytope $2 \Delta \subseteq 2 [\Delta^*]^*$. If $\{v_1, \ldots, v_l\} \subset N$ is
the set of vertices of $[\Delta^*]$ then the polytope $[\Delta^*]^*$ is determined by the inequalities
$\langle x, v_i \rangle \geq -1$ $(1 \leq i \leq l)$. The  interior lattice point  $u_0 \in 2 [\Delta^*]^*$ must
safisfy the inequalities $\langle u_0, v_i \rangle > -2$ $(1 \leq i \leq l)$. Since
$\langle u_0, v_i \rangle >  \in \Z $ $ (1 \leq i \leq l)$, we obtain $\langle u_0, v_i \rangle \geq -1$ $(1 \leq i \leq l)$, i.e., $u_0$ is a nonzero lattice point in $[\Delta^*]^*$.
Since $\Delta = [[\Delta^*]^*]$, $u_0$  is a nonzero lattice point contained in $\Delta$ and
in the $k$-dimensional cone  $\R_{\geq 0} \Theta$ over $\Theta$. On the other hand,
$\Delta \cap R_{\geq 0} \Theta = \Theta \subset \Pi_\Theta$. Contradiction.
\end{proof}

\begin{prop} \cite{Mav13} \label{reg-dual}
Let $\Theta \prec \Delta$ be an ordinary  $k$-dimensional face of a pseudoreflexive polytope
$\Delta$ such
$\dim [\Theta^*] = \dim \Theta^* =d-1 -k \geq 0$. Then one has $[[\Theta^*]^*] =
\Theta$.
\end{prop}

\begin{proof}
If $\dim \Theta^* = \dim [\Theta^*] =d-k-1$ then there exists a point $y \in \Delta^*$  which
is contained  in  the
relative interior of $[\Theta^*]$ and in the relative interior of  $\Theta^*$. In particular, $y \in \Theta^*$
is contained in the relative
interior of the $(d-k)$-dimensional normal cone
$\sigma^\Theta$ and therefore the minimum of
the linear function
$\langle *, y \rangle$ on $\Delta$ equals $-1$ and it is attained on
$\Theta = \{ x \in \Delta\,:\, \langle x, y \rangle =-1 \}$. By definition of the polar polytope
$[\Delta^*]^*$,
the minimum of  $\langle *, y \rangle$  on
$[\Delta^*]^*$ also equals $-1$,  and
it is attained on the $k$-dimensional
dual face $[\Theta^*]^* \prec [\Delta^*]^*$. Hence,
$[\Theta^*]^*$ contains  the lattice face $\Theta$ and $[[\Theta^*]^*]$ also contains $\Theta$. By \ref{face1},
 the lattice polytope $[[\Theta^*]^*]$ is face of  $[[\Delta^*]^*]=\Delta$
of dimension $\leq k$. Since   $[[\Theta^*]^*]$
contains the $k$-dimensional face $\Theta \prec \Delta$,
the face $[[\Theta^*]^*] \prec \Delta$ must be $\Theta$.
\end{proof}

\begin{definition}
We call a $k$-dimensional face $\Theta$
of a pseudoreflexive polytope $\Delta \subset M_\R$ {\em regular}, if
\[ \dim [\Theta^*] = d-k-1. \]
A $k$-dimensional face $\Theta$ is called {\em singular} if it is not regular.
\end{definition}

By \ref{reg-dual}, we immediately obtain:

\begin{cor} \label{reg-dual2}
Let $\Delta \subset M_\R$ be a $d$-dimensional pseudoreflexive polytope. Then
there exists a natural bijection $\Theta \leftrightarrow \Theta^\vee :=[\Theta^*]$ between the set of
$k$-dimensional regular faces of $\Delta$ and $(d-k-1)$-dimensional regular faces
of $\Delta^*$.
\end{cor}

\begin{remark}
Pseudoreflexive lattice polytopes $\Delta$ satisfy a combinatorial duality
$\Delta \leftrightarrow   \Delta^\vee$ that extends the
polar duality for reflexive lattice
polytopes .  However, in contrast to polar duality for reflexive polytopes
there is no natural bijection between arbitrary $k$-dimensional faces of a
pseudoreflexive polytope $\Delta$ and $(d-k-1)$-dimensional faces
of its dual $\Delta^\vee$.
Such a natural bijection exists only for regular  $k$-dimensional
faces $\Theta \subset \Delta$.
\end{remark}

\begin{remark} \label{reg-ord}
By \ref{face1}, every regular face $\Theta \prec \Delta$ is ordinary.
It is easy to see that for a $(d-1)$-dimensional face $\Theta \prec \Delta$ the following conditions
are equivalent:

(i) $\Theta$ is regular;

(ii) $\Theta$ is ordinary;

(iii) the  integral distance from $0 \in M$ to $\Theta$ is $1$.
\end{remark}

\begin{remark}
Reflexive polytopes of dimension $3$ and $4$ have been classified by Kreuzer und
Skarke \cite{KS98,KS00}.
It is natural task to extend these classification to lattice polytopes with Fine
 interior $0$. By \ref{FI0-equiv}, a lattice polytope $\Delta$
of dimension $3$ or $4$ has Fine
 interior $0$ if and only if $\Delta$ contains $0$ in its interior
and $\Delta$ is contained
 in some reflexive polytope $\Delta'$. 

All $3$-dimensional lattice polytopes with the single
 interior lattice point $0$ have been classified by Kasprzyk \cite{Kas10}. 
There exists exactly 674,688 3-dimensional lattice polytopes $\Delta$ 
with only a single
interior lattice point.  However,  not all these polytopes $\Delta$ have Fine 
interior $0$.  I was informed by  Kasprzyk that among these $674,688$ lattice 
polytopes there exist exactly $9,089$ lattice polytopes whose Fine interior has
dimension $\geq 1$.  These polytopes correspond to elliptic surfaces, 
Todorov surfaces and some other interesting algebraic surfaces. 

According to Kreuzer und
Skarke \cite{KS98,KS00}, there exist exactly $4,319$ $3$-dimensional reflexive 
polytopes. 
 We remark that
 canonical models of $K3$-surfaces coming from $3$-dimensional
reflexive polytopes
 have at worst toroidal quotient singularities of type $A_n$. However,
the canonical
 models of $K3$-surfaces coming from $3$-dimensional lattice polytopes $\Delta$ 
with the weaker condition
 $\Delta^{FI} =0$ may have more general Gorenstein canonical
singularities of types $D_n$ and $E_n$.

 Analogously, we remark that canonical
singularities of $3$-dimensional Calabi-Yau varieties obtained
 as hypersurfaces in $4$-dimensional Gorenstein toric
Fano varieties defined by $4$-dimensional
reflexive polytopes are toroidal. They admit smooth crepant resolutions,
because any $3$-dimensional
 $\Q$-factorial terminal Gorenstein toric variety is smooth.
Singularities of $3$-dimensional Calabi-Yau varieties $X$ obtained as minimal models of $\Delta$-nondegenerate hypersurfaces with $\Delta^{FI} =0$
 generally can not be  resolved crepantly, because $\Q$-factorial Gorenstein terminal singularities
 in dimension $3$ are $cDV$-points that may cause that the stringy Euler number of $X$ will be
 a rational number \cite{DR01}.
So the classification of $4$-dimensional
lattice polytopes $\Delta$ with Fine interior $0$
 would give many new examples of $3$-dimensional
Calabi-Yau varieties with isolated  terminal $cDV$-points
 that need additionally to be smoothed by a deformation \cite{Na94}
in order to get a smooth Calabi-Yau $3$-fold. 

It would be  very interesting to know  what 
rational numbers can appear as stringy Euler numbers of minimal
$3$-dimensional Calabi-Yau varieties coming from $4$-dimensional lattice 
polytopes $\Delta$ with $\Delta^{FI} = 0$.   
\end{remark}

\section{The stringy Euler number}

\begin{definition}
If $V$ is a smooth projective algebraic variety over $\C$, then
its {\em $E$-polynomial} (or {\em Hodge polynomial})
is defined as $$E(V; u, v) :=
\sum_{0 \leq p,q \leq  \dim V} (-1)^{p+q} h^{p,q}(V) u^p v^q,$$
where $h^{p,q}(V)$ are Hodge numbers of $V$.

 For any quasi-projective
variety $W$ one can use the mixed Hodge structure in $k$-th cohomology
group $H^{k}_c(W)$
with compact supports and define $E(W; u,v)$ by the formula
\[ E(W; u,v):=  \sum_{ p,q } e^{p,q} (W) u^p v^q, \]
where the coefficients
\[ e^{p,q}(W) = \sum_k (-1)^k h^{p,q}(H^k_c(W)) \]
are called {\em Hodge-Deligne numbers} of $W$ \cite{DKh86}.
\end{definition}

\begin{definition}
\label{def:str}
Let $X$ be a normal projective variety over $\Cd$ with
at worst $\Qd$-Gorenstein log-terminal
singularities. Denote by $r$ the minimal positive integer such
that $rK_X$ is a Cartier divisor.
 Let $\rho := Y \to X$ be a log-desingularization
together with smooth irreducible
divisors $D_1, \dots, D_k$ with simple normal crossings whose support
covers the exceptional locus of $\rho$.
We can uniquely write
\begin{align*}
K_Y = \rho^*K_X + \sum_{i=1}^k a_iD_i \text{.}
\end{align*}
for some rational numbers $a_i \in \frac{1}{r}\Z$ satisfying
 the additional condition $a_i = 0$ if $D_i$ is not
in the exceptional locus of $\rho$.
We set $I := \{ 1, \ldots, k\}$ and,
for any $\emptyset \subseteq J \subseteq I$, we define
\begin{align*}
D_J := \begin{cases}Y &\text{if $J = \emptyset$,}\\
\bigcap_{j \in J} D_j &\text{if $J \ne \emptyset$,}\end{cases}
&&
D_J^\circ := D_J \setminus \bigcup_{j \in I \setminus J} D_j\text{.}
\end{align*}
\begin{align*}
E_{\rm str}(X; u,v) & :=
\sum_{\emptyset \subseteq J \subseteq I}
 \rleft(\prod_{j \in J}\frac{uv-1}{(uv)^{a_j+1}-1}\rright)
\cdot E(D_J^\circ; u,v)\\
&=
\sum_{\emptyset \subseteq J \subseteq I}
\rleft(\prod_{j \in J}\frac{uv-1}{(uv)^{a_j+1}-1}-1\rright) \cdot E(D_J; u,v),
\end{align*}
where $E(D_J; u,v) = \sum_{p,q} (-1)^{p+q}h^{p,q}(D_J) u^p v^q$ is the $E$-polynomial
of the smooth projective variety $D_J$.
The rational function $E_{\rm str}(X; u, v)$ is called {\em stringy $E$-function} of
the algebraic  variety $X$.

Let $x \in X$ be a point on $X$. We define the {\em local stringy $E$-function of $X$ at $x \in X$} by the formula
\[ E_{\rm str}(X, x; u, v) := \sum_{\emptyset \subseteq J \subseteq I}
\rleft(\prod_{j \in J}\frac{uv-1}{(uv)^{a_j+1}-1}-1\rright) \cdot E(\rho^{-1}(x)  \cap D_J; u,v). \]
In particular, we define the  {\em local stringy Euler number} of $X$ at point $x \in X$ as
 \[ e_{\rm str}(X, x) := \sum_{\emptyset \subseteq J \subseteq I}
\rleft(\prod_{j \in J}\frac{-a_j}{a_j+1}-1\rright) \cdot e(\rho^{-1}(x)  \cap D_J). \]
\end{definition}

Our goal is  to derive a combinatorial formula for  the stringy $E$-function $E_{\rm str}(X; u,v)$
of a minimal Calabi-Yau model $X$ of an affine $\Delta$-nondegenerate hypersurface
$Z \subset \T_d$ corresponding to a $d$-dimensional
lattice polytope $\Delta \subset M_\R$ such that $\Delta^{FI} =0$. For this purpose we need a rational function $R(C,m, t)$
associated with an arbitrary $d$-dimensional rational polyhedral cone $C \subset  N_\R$ and a primitive lattice
point $m \in M$ in the interior of the dual cone $C^* \subset M_\R$.

  \begin{definition} \label{r-func}
Let $C \subset N_\R$ be an arbitrary $d$-dimensional rational polyhedral
cone with vertex $0 = C \cap (- C)$ and
let $m \in M$ be a primitive lattice point such that
$C(1):= \{ y \in  C \; : \; \langle m, y \rangle  \leq 1\}$
is a $d$-dimensional compact polytope with rational vertices.
Let  $C^\circ$ be the interior
of the cone $C$.
 We define two power series
\[ R(C, m,  t) := \sum_{ n \in C  \cap N} t^{\langle m, n \rangle} \]
and
\[ R(C^\circ, m, t):=  \sum_{ n \in C^\circ \cap N} t^{\langle m, n \rangle}.\]
\end{definition}

\begin{example} If $M = N = \Z^d$,  $C = \R_{\geq 0}^d \subset \R^d$
and $m=(1, \ldots, 1)$. Then
$C(1)$ is a $d$-dimensional simplex in $\R^d$ defined by the conditions
$x_i \geq 0 \; (1 \leq i \leq d)$,
$\sum_{i=1}^d x_i \leq 1$. We have
\[ R(C, m,  t) = \left( \sum_{k =0}^\infty t^k \right)^d = \frac{1}{(1-t)^d} \]
and
\[ R(C^\circ, m,  t) = \left( \sum_{k =1}^\infty t^k \right)^d = \frac{t^d}{(1-t)^d} \]
\end{example}

\begin{prop} \label{lim-con}
The power series   $R(C, m, t)$ and $R(C^\circ, m, t)$ are rational functions satisfying the
the equation
\[ R(C, m,  t) = (-1)^d R(C^\circ, -m, t). \]
Moreover, two limits
\[ \lim_{t \to 1} (1-t)^d R(C, m, t), \;\;  \lim_{t \to 1} (t-1)^d R(C^\circ, -m, t)\]
equal the integral volume $v(C(1)) = d! Vol(C(1))$, where
$Vol(C(1))$ denotes the usual volume of the $d$-dimensional
compact set $C(1)$.
\end{prop}

\begin{proof}
First we remark that $R(C, m, t)$ is  a rational function, because
$R(C, m, t)$ can be considered as the Poincar\'e series
of the graded finitely graded commutative
semigroup algebra $\C[ C \cap N]$ such that  the degree
of an element $n \in C \cap N$ equals $\langle m, n \rangle$.
The function $R(C^\circ, m, t)$
is also rational, because it is the Poincar\'e series
of a graded homogeneous ideal in  $\C[ C \cap N]$.

In order to compute the rational functions
$R(C, m, t)$ and $R(C^\circ, m, t)$ explicitly
we use a regular simplicial
subdivision of the cone $C$  defined by  a finite fan $\Sigma = \{\sigma \}$
consisting of cones $\sigma$ generated by parts of $\Z$-bases of $N$.   Denote by    $\sigma^\circ$
the relative interior of a
cone $\sigma \in \Sigma$.  Then we obtain
\begin{align} \label{c}
 R(C,m,  t) = \sum_{\sigma \in \Sigma}
R(\sigma^\circ, m, t)
\end{align}
and
\begin{align} \label{c-circ}
 R(C^\circ,m,  t) = \sum_{\sigma \in \Sigma \atop \sigma^\circ \subseteq C^\circ}
R(\sigma^\circ, m, t).
\end{align}
If $\sigma \in \Sigma$ is a
$k$-dimensional cone, the semigroup $\sigma \cap N$ is freely generate by some
elements $v_1, \ldots, v_k \in N$  such that $\langle m,  v_i \rangle = c_i  \in \Z_{>0}$ $(1 \leq i \leq k)$. Therefore,  we obtain
\[ R(\sigma,m,  t)  = \prod_{i=1}^k \frac{ 1}{ 1 - t^{c_i}} \]
and
\[ R(\sigma^\circ,-m,  t)  = \prod_{i=1}^k \frac{ t^{-c_i}}{ 1 - t^{-c_i}} =
\prod_{i=1}^k \frac{ 1}{t^{c_i} -1} = (-1)^k  \prod_{i=1}^k \frac{ 1}{1-t^{c_i}} =  (-1)^k  R(\sigma,m,  t). \]
In order to prove the equation  $R(C, m,  t) = (-1)^d R(C^\circ, -m, t)$  for the whole $d$-dimensional
cone $C$ we note that for any $\sigma \in \Sigma$ one has
\[ R(\sigma^\circ,m,  t) = \sum_{\tau \preceq \sigma}
(-1)^{\dim \sigma - \dim \tau}  R(\tau, m,  t) =
 \sum_{\tau \preceq \sigma}  (-1)^{\dim \sigma}  R(\tau^\circ, -m,  t) . \]
Using  the equalities (\ref{c}) and (\ref{c-circ}), we get
\begin{align*}
 R(C^\circ, m,  t) =& \sum_{\sigma \in \Sigma \atop \sigma^\circ \subseteq C^\circ}
R(\sigma^\circ, m, t) = \sum_{\sigma \in \Sigma \atop \sigma^\circ \subseteq C^\circ}
\sum_{\tau \preceq \sigma}  (-1)^{\dim \sigma}  R(\tau^\circ, -m,  t) = \\= &
\sum_{\tau \in \Sigma} R(\tau^\circ, -m,  t) \sum_{\tau \preceq \sigma \in \Sigma} (-1)^{\dim \sigma} = \\ = &
(-1)^d
\sum_{\tau \subseteq C} R(\tau^\circ, -m,  t) = (-1)^d R(C, -m ,t),
\end{align*}
because for any cone $\tau \in \Sigma$ one has
$\sum_{\tau \preceq \sigma } (-1)^{\dim \sigma} = (-1)^{ d} $.

We note that the limit
\[  \lim_{t \to 1} (1-t)^d R(\sigma^\circ,m,  t) = \lim_{t\to 1 } (1 - t)^d
\prod_{i=1}^k
\frac{ t^{c_i}}{ 1 -t^{c_i}} \]
is zero if $k = \dim \sigma < d$.
If $\sigma \in \Sigma(d)$ is a $d$-dimensional cone, then
 \[  \lim_{t \to 1} (1-t)^d R(\sigma^\circ,m ,  t) =
\lim_{t\to 1 } (1-t)^d \prod_{i=1}^d \frac{ t^{c_i}}{1 -  t^{c_i}} =
 \prod_{i=1}^d \frac{ 1}{ {c_i} } =d! Vol( \sigma(1)), \]
because $ \sigma(1)$ is a $d$-dimensional simplex which is  the
convex hull of vectors $\frac{1}{c_i}v_i$, where $v_1, \ldots, v_d$ is a $\Z$-basis of
$M$.
 Using (\ref{c}),  we get
 \[  \lim_{t \to 1} (1-t)^d R(C,m,  t) = \sum_{\sigma \in \Sigma}  \lim_{t \to 1} (1-t)^d
R(\sigma^\circ, m, t)  =  \sum_{\sigma \in \Sigma(d)} v(\sigma(1)) =
 v(C(1)) ,  \]
because
 \[ Vol( C(1)) = \sum_{\sigma \in \Sigma(d)} Vol(\sigma(1)).   \]
It follows now from the equation  $R(C, m,  t) = (-1)^d R(C^\circ, -m, t)$ that
 \[  \lim_{t \to 1} (t-1)^d R(C^\circ, -m,  t) = v (C(1)). \]
\end{proof}


The following results of Danilov and Khovanskii allow us to compute
the polynomial $E(Z; u, v)$ for any $(d-1)$-dimensional
$\Delta$-nondegenerate affine hypersurface $Z_\Delta \subset
T$ \cite[Remark 4.6]{DKh86}.

\begin{theorem} \label{dkh86}
Let $\Delta \subset M_\R$ be a $d$-dimensional lattice polytope. The power series
\[ P(\Delta, t) := \sum_{k =0}^{\infty}  |k\Delta \cap M| t^k \]
is a rational function of the form
\[ P(\Delta, t) = \frac{ \psi_0(\Delta)  + \psi_1(\Delta) t +
\cdots + \psi_d(\Delta) t^d}{(1 -t)^{d+1}}, \]
where $\psi_i(\Delta)$ $(0 \leq i \leq d)$ are nonnegative integers satisfying the
conditions $\psi_0(\Delta) =1$, $\sum_{i=1}^d \psi_i(\Delta) = v(\Delta)$.

Let $E(Z_\Delta; u, 1)$ be the $E$-polynomial
of a $(d-1)$-dimensional $\Delta$-nondegenerate hypersurface
$Z_\Delta \subset \T_d$. Then one has
\[ E(Z_\Delta; u, 1) = \frac{(u -1)^d - (-1)^{d}}{u} +
(-1)^{d-1} \sum_{i=1}^d \psi_i(\Delta) u^{i-1}. \]
In particular, the Euler number $e(Z_\Delta) = E(Z_{\Delta}; 1,1)$ equals
$(-1)^{d-1} v(\Delta)$.
\end{theorem}

In particular, one obtains

\begin{cor} \cite{Kho78} \label{kh78}
 The Euler number $e(Z_\Delta) = E(Z_{\Delta}; 1,1)$ of a $(d-1)$-dimensional
 affine
$\Delta$-nondegenerate hypersuface $Z_\Delta$ equals
$(-1)^{d-1} v(\Delta)$.
\end{cor}

\begin{definition} \label{def:e-theta}
Let  $\Delta \subset M_\R$ is an arbitrary  $d$-dimensional almost pseudoreflexive
polytope. For any $k$-dimensional face $\Theta \prec \Delta$
we define the polynomial
\begin{align*}
 E(\Theta, u) := & \frac{( u-1)^{\dim \Theta}}{u} +   \frac{( u-1)^{\dim \Theta +1}}{u}
\sum_{l \in \Z_{\geq 0}} |l\Theta \cap M| u^l \\
= &  \frac{(u -1)^k - (-1)^{k}}{u} +
(-1)^{k-1} \sum_{i=1}^k \psi_i(\Theta) u^{i-1}.
\end{align*}
\end{definition}

\begin{definition}
Let  $\Delta \subset M_\R$ be  an arbitrary $d$-dimensional almost pseudoreflexive  polytope and
let  $\sigma^\Theta$
be the $(d-k)$-dimensional cone in  the normal fan $\Sigma^\Delta$  that correspond
to the $k$-dimensional face $\Theta \prec \Delta$. We choose an
arbitrary lattice point $m \in \Theta \cap M$ and
set
\[ R(\sigma^\Theta, u):= R(\sigma^\Theta, -m, u) \]
where the rational
function $R(\sigma^\Theta, -m, u)$  defined  in
\ref{r-func}. It is easy to see that the rational
function $R(\sigma^\Theta, -m, u)$ does not depend on the choice of the lattice point
$m \in \Theta \cap M$.
\end{definition}

We prove the following theorem:

\begin{theorem} \label{E-f-u}
Let $\Delta$ be an arbitrary $d$-dimensional almost pseudoreflexive  polytope.
Denote by  $X$ a canonical $(d-1)$-dimensional Calabi-Yau model of a $\Delta$-nondegenerate
hypersurface $Z_\Delta \subset T$ in the $d$-dimensional algebraic torus $T$.
 Then
the stringy function $E_{\rm st}(X; u, 1)$
can be computed as follows:
\[ E_{\rm st}(X; u, 1) =
\sum_{\Theta \preceq \Delta \atop \dim \Theta \geq 1} E(\Theta, u)
\cdot R( \sigma^{\Theta}, u)\cdot (1-u)^{d - \dim \Theta}. \]
\end{theorem}

\begin{proof}
Let $\widehat{\Sigma}$ be a common  regular simplicial subdivision of the normal fans $\Sigma^\Delta$ and
$\Sigma^{\Delta^{can}}$.  As in Theorem \ref{MMP-nondeg} we obtain
birational
morphisms $\rho_1\, :\,  \widehat{Z}_\Delta \to \overline{Z}_\Delta$,
$\rho_2\, :\,
\widehat{Z}_\Delta \to X$ in  the diagram
\[
\xymatrix{   & \widehat{Z}_\Delta\ar[ld]_{\rho_1} \ar[rd]^{\rho_2} & \\
\overline{Z}_\Delta \ar@{-->}[rr]^\alpha &  & X}
\]
where $\widehat{Z}_\Delta$ is a smooth variety obtained as
Zariski closure of $Z_\Delta$ in the smooth projective
toric variety $\widehat{P}$ defined by the fan $\widehat{\Sigma}$. For computing
the stringy $E$-function $E_{\rm str}(X, u,1)$ we use the formula (\ref{def:str})
\[ E_{\rm str}(X; u, 1) =
\sum_{\emptyset \subseteq J \subseteq I}
 \rleft(\prod_{j \in J}\frac{u-1}{u^{a_j+1}-1}\rright)
\cdot E(D_J^\circ; u,1) \]
where the strata $D_J^\circ$ are intersections of the hypersurface $\widehat{Z}_\Delta \subset
\widehat{P}$ with torus orbits corresponding to cones $\sigma \in \widehat{\Sigma}$.

Let $\sigma \in \widehat{\Sigma}(k)$ be a $k$-dimensional simplicial cone generated by primitive
lattice vectors $v_{i_1}, \ldots, v_{i_k}$. Then the relative interior $\sigma^\circ$ of $\sigma$
is contained in the relative interior $\sigma_\circ^\Theta$ of some cone $\sigma^\Theta$ of the normal fan $\Sigma^\Delta$ corresponding to a face $\Theta \prec \Delta$, $\dim \Theta \leq d -k$. We set
$J := \{i_1, \ldots, i_k \}$.
By \ref{fi-cy}, the discrepancy
coefficients $a_j$ of smooth divisors $D_j$ $(j \in J)$ on $\widehat{Z}_\Delta$  can be computed
by the formula $a_j = - \langle m , v_j \rangle - 1$  $(j \in J)$ where $m \in M$ is any lattice point in the face $\Theta$. Since the fiber $\rho_1^{-1}(p)$ of the birational toric morphism $\rho_1$ over every point $p \in \P_\Delta$ consists of torus orbits, the  codimension $k$ stratum $D_J^\circ$  is
isomorphic to the product of a torus $(\C^*)^{d-k - \dim \Theta}$
 and a
 $\Theta$-nondegenerate hypersurface $Z_\Theta \subset (\C^*)^{\rm dim \Theta}$. By
 \ref{dkh86} and \ref{def:e-theta} we
 have
 \[ E(D_J^\circ; u, 1) = E(\Theta, u) \cdot (u-1)^{d-k-\dim \Theta} \]
 and
 \begin{align*}  \rleft(\prod_{j \in J}\frac{u-1}{u^{a_j+1}-1}\rright)
\cdot E(D_J^\circ; u,1)  = &  \rleft(\prod_{j \in J}\frac{u-1}{u^{\langle - m, v_j \rangle}-1}\rright)
\cdot E(\Theta, u) \cdot (u-1)^{d-k-\dim \Theta} = \\ = &
(-1)^{\dim \sigma} R(\sigma, -m, u) \cdot  E(\Theta, u) \cdot (u-1)^{d -\dim \Theta} .
\end{align*}
Using \ref{lim-con}, we obtain $(-1)^{\dim \sigma} R(\sigma, -m, u) = R(\sigma^\circ, m, u) $ and
\[
\sum_{\sigma \in \widehat{\Sigma} \atop \sigma^\circ \subseteq \sigma^\Theta_\circ}
R(\sigma^\circ, m, u) = R(\sigma^\Theta_\circ, m, u) = (-1)^{\dim \sigma^\Theta} R(\sigma^\Theta, -m, u). \]
Since $\dim \sigma^\Theta = d - \dim \Theta$,  we conclude
\begin{align*} E_{\rm str}(X; u, 1)= &
\sum_{\Theta \preceq \Delta} E(\Theta, u) \cdot (u-1)^{d - \dim \Theta} \cdot
\sum_{\sigma \in \widehat{\Sigma} \atop \sigma^\circ \subseteq \sigma^\Theta_\circ}
R(\sigma^\circ, m, u) = \\= &\sum_{\Theta \preceq \Delta} E(\Theta, u) \cdot (1-u)^{d - \dim \Theta} R(\sigma^{\Theta}, -m, u).
\end{align*}
\end{proof}

\begin{theorem} \label{E-f-u1}
Let $\Delta \subset M_\R$  be an arbitrary $d$-dimensional almost pseudoreflexive polytope.
Denote by  $X$ a canonical Calabi-Yau model of a $\Delta$-nondegenerate affine
hypersurface $Z_\Delta \subset \T_d$ in the $d$-dimensional algebraic torus $\T_d$.
Then
\[    e_{\rm st}(X) = \sum_{\Theta \preceq \Delta \atop \dim \Theta \geq 1}
(-1)^{\dim \Theta -1} v(\Theta) \cdot v(\sigma^{\Theta} \cap \Delta^*), \]
where $\sigma^\Theta$ is the cone in the normal fan of the polytope $\Delta$  and
$\Delta^* \subset N_\R$ is the polar polytope of $\Delta$.
\end{theorem}

\begin{proof}
One has
\begin{align*}  e_{\rm st}(X) =&  \lim_{u \to 1}  E_{\rm st}(X; u, 1) =
\sum_{\Theta \preceq \Delta \atop \dim \Theta \geq 1} E(\Theta, 1)
\lim_{u \to 1}  R( \sigma^{\Theta}, u)\cdot (1-u)^{d - \dim \Theta}.
\end{align*}

It remains to apply Corollary \ref{kh78}
\[  E(\Theta, 1) = (-1)^{\dim \Theta-1}v(\Theta) \]
and Proposition \ref{lim-con}
\[ \lim_{u \to 1}  R( \sigma^{\Theta}, u)\cdot (1-u)^{d - \dim \Theta} =
v(\sigma^\Theta \cap \Delta^*). \]
\end{proof}

\begin{remark}
It is easy to see that the formula for the stringy Euler number in Theorem \ref{E-f-u1} is a generalization
of the formula (\ref{e-reflex}) in the case when $\Delta$ is a reflexive polytope. If $\Theta \prec \Delta$ is a $(d-k)$-dimensional
face of reflexive polytope $\Delta$
then the $k$-dimensional
polytope $\sigma^{\Theta} \cap \Delta^*$ is a lattice pyramid with height $1$
over the $(k-1)$-dimensional dual face $\Theta^*$ of the polar
reflexive polytope $\Delta^*$. Therefore, $v( \sigma^{\Theta} \cap \Delta^*) = v(\Theta^*)$.
On the other hand, the $d$-dimensional reflexive polytope $\Delta$ is the union of $d$-dimensional
pyramids over all $(d-1)$-dimensional faces $\Theta \prec \Delta$. So we have
\[ v(\Delta) = \sum_{\Theta \preceq \Delta \atop \dim \Theta = d- 1} v(\Theta). \]
Thus, we obtain
\[ \sum_{\Theta \preceq \Delta \atop \dim \Theta \geq 1}
(-1)^{\dim \Theta -1} v(\Theta) \cdot v(\sigma^{\Theta} \cap \Delta^*) = \sum_{k=1}^{d-2}  (-1)^{k-1} \sum_{\Theta \preceq \Delta \atop \dim \Theta =k}
v(\Theta) \cdot v\left(\Theta^*\right)
.  \]
\end{remark}

We consider below several examples illustrating applications of our formula to
non-reflexive polytopes $\Delta$.

\begin{example} \label{quint1}
The Newton polytope $\Delta$ of a general $3$-dimensional quintic $X$ in
$\P^4$ containing the
point $(1:0:0:0:0) \in \P^4$ is the almost reflexive polytope
\[  \Delta = \{ (x_1,x_2,x_3, x_4) \in \R^4_{\geq 0}\; :\;
1 \leq x_1 + x_2 + x_3 + x_4 \leq 5. \},  \]
which is a set-theoretic difference
of two $4$-dimensional simplices.
The Fine interior of $\Delta$ consists of the
single lattice point $p =(1,1,1,1)$.
The integral distance between $p$ and the
$3$-dimensional face $\Theta$ defined by the equation
$x_1 + x_2 + x_3 + x_4 = 1$ equals $3$. Therefore,
$\Delta$ is not a reflexive polytope. One has
$v(\Delta) = 5^4 - 1^4 = 624$. The polytope
$\Delta$ has $6$ faces $\Theta$ of codimension $1$ ($4$ faces
 $\Theta$ with  $v(\Theta) = 5^3 - 1^3 = 124$, one face
$\Theta$ with $v(\Theta) =5^3$ and one
  face $\Theta$ with $v(\Theta) = 1^3$).  There also $14$ faces of dimension
$2$ and $16$ faces of dimension $1$ in $\Delta$.
Our formula for the stringy Euler number gives:
\begin{align*}
  e_{\rm str}(X)= & -(5^4 - 1^4) + 4 (5^3 - 1^3) + 5^3 + 1^3 \cdot \frac{1}{3}
\\- & 4 \cdot 5^2 -
 6 \cdot (5^2 - 1^2) - 4 \cdot 1^2    \cdot \frac{1}{3} + 6 \cdot 5 + 4 \cdot (5-1) +
 6 \cdot 1 \cdot \frac{1}{3} = -200.
\end{align*}
This is a  well-known fact, since $X$ is a smooth quintic $3$-fold.
\end{example}

\begin{example} \label{quint2}
The Newton polytope $\Delta$ of a general $3$-dimensional quintic in
$\P^4$ having an isolated quadratic (conifold) singularity
at point $x=(1:0:0:0:0) \in \P^4$ is the almost pseudoreflexive polytope
\[  \Delta = \{ (x_1,x_2,x_3, x_4) \in \R^4_{\geq 0}\; :\;
2 \leq x_1 + x_2 + x_3 + x_4 \leq 5, \}
  \]
which is again  a set-theoretic difference
of two $4$-dimensional simplices.
The Fine interior of $\Delta$ consists of the
single lattice point $p =(1,1,1,1)$.
The integral distance between $p$ and the
$3$-dimensional face $\Theta$ defined by the equation
$x_1 + x_2 + x_3 + x_4 = 2$ equals $2$. Therefore,
$\Delta$ is not a reflexive polytope. One has
$v(\Delta) = 5^4 - 2^4 = 609$. The polytope $\Delta$ has $6$ faces $\Theta$
of codimension $1$ ($4$ faces
 $\Theta$ with  $v(\Theta) = 5^3 - 2^3 = 117$,
one face $\Theta$ with $v(\Theta) =5^3$ and one
  face $\Theta$ with $v(\Theta) = 2^3$).
There also $14$ faces of dimension $2$ and $16$ faces of dimension $1$ in $\Delta$. Our formula for the stringy Euler number gives:
\begin{align*}
  e_{\rm str}(X)= & -(5^4 - 2^4) + 4 (5^3 - 2^3) + 5^3 + 2^3 \cdot \frac{1}{2}
\\- & 4 \cdot 5^2 -
 6 \cdot (5^2 - 2^2) - 4 \cdot 2^2    \cdot \frac{1}{2} + 6 \cdot 5 + 4 \cdot (5-2) +
 6 \cdot 2 \cdot \frac{1}{2} = -198. 
\end{align*}
Unfortunatly, this singular Calabi-Yau $3$-fold $X$ does not have a projective 
small resolution of its singularity.  So $X$ has no a smooth projective 
birational Calabi-Yau model. 
 \end{example}

\begin{remark} \label{isol-sing}
If a log-desingularization $\rho\,:\, Y \to X$
of a projective variety $X$ contains only one smooth
exceptional divisor $D$ such that $K_Y = \rho^*K_X + a D$, then
\[ e_{\rm st}(X)  = e(Y) + e(D) \left( \frac{-a}{a+1} \right). \]
In particular, $ e_{\rm st}(X)$ is not an integer, if $ (a+1)$ does not divide
$e(D)$.
\end{remark}

\begin{example}
 One can generalize Example \ref{quint2} and compute the stringy Euler number of a general $d$-dimensional Calabi-Yau hypersurface
$X' \subset \P^{d+1}$ of degree $d+2$
 with a single quadratic singularity at point $x=(1: 0 :\cdots :0)$. The blow up of this
 singular point is a desingularization $\rho\, :\, \widehat{X'} \to X'$ such that the exceptional
 divisor $D \subset \widehat{X'}$ is isomorphic to a
 $(d-1)$-dimensional quadric.  One has
 \[ K_{\widehat{X'}} = K_{X'} + (d-2)D.  \]
By \ref{isol-sing}, we obtain
 \[ e_{\rm str}(X' )  = e( \widehat{X'})  - \frac{d-2}{d-1}e(D). \]
 Therefore, the local stringy Euler number of the singular point $x \in X'$ equals
 \[ e_{\rm str}(X', x) =  e(D) -  \frac{d-2}{d-1}e(D) = \frac{e(D)}{d-1} \]
If the dimension $d \geq 4$ is an even number then $e(D) = d$ and   $e_{\rm str}(X' )  =
 \frac{c}{d-1} \in \Q \setminus \Z$ for some coprime numbers $c, d-1$. In particular, $e_{\rm str}(X' )$ is not an integer.
 \end{example}

So far no mirror manifolds have been known for singular Calabi-Yau varieties $X$
with non-integral stringy Euler number $e_{\rm st}(X) \in \Q \setminus \Z$.

\section{Calabi-Yau hypersurfaces
in $\P(a,1,\ldots,1)$}

Let $a,b \in \N$ be two integers $a, b \geq 2$. We put $d := ab+l$ for some integer
$1 \leq l \leq a-1$
and consider Calabi-Yau hypersurfaces of degree $a+d$ in
the  weighted projective space
$$\P(a, 1^d) := \P(a, \underbrace{1, \ldots, 1}_d)$$
of dimension $d \geq 5$.
The space  of
quasihomogeneous polynomials of degree $a+d$ in $d+1$ variables
$z_0, z_1, \ldots, z_d$ ($\deg z_0 = a$, $\deg z_i =1 $, $1 \leq i \leq d)$
has the  monomial basis $z_0^{m_0} z_1^{m_1} \cdots z_d^{m_d}$
determined by the lattice points $(m_0, m_1, \ldots, m_d) \in
\Z^{d+1}_{\geq 0}$ satisfying the condition
\[ am_0 + m_1 + \cdots m_d = a+d. \]

The convex set
\[ S_d := \{(x_0, x_1, \ldots, x_d) \in
\R^{d+1}_{\geq 0} \; : \;  ax_0 + x_1 + \cdots x_d = a +d \}. \]
is a $d$-dimensional simplex having
the unique interior lattice point $p:= (1,\ldots,1) \in \Z^{n+1}$,
$d$ integral vertices $\nu_i =
(0, \ldots,\underbrace{ a +d}_i, \ldots , 0)$ $1 \leq i \leq d$ and
one  rational vertex $\nu_0 :=
( \frac{a+d}{a}, 0, \ldots, 0)$.
The convex hull of the set $S_d \cap \Z^{d+1}$ is  the lattice polytope
$\Delta$ which is the intersection of the simplex $S_d$ with the half-space
define by the inequality $x_0 \leq b+1$. The lattice polytope
$\Delta $ has $2d$ vertices, first $d$ vertices of $\Delta$ belong
to the hyperplane $x_0 = 0$ and the remaining $d$ vertices of $\Delta$ belong
to the hyperplane $x_0 = b+1$. The lattice polytop $\Delta$ is not reflexive, because
the integral distance between its single interior lattice point $p$ and
the $(d-1)$-dimensional face
in the hyperplane  $x_0 = b+1$
is equal to $b \geq 2$. However, it is easy to show that $\Delta$ is a pseudoreflexive polytope.
Its dual pseudoreflexive
polytope $\Delta^\vee = [ \Delta^*]$ is a $d$-dimensional lattice simplex
whose vertices $v_0, v_1, \ldots, v_d \in N$ satisfy the
relation $a v_0 + \sum_{i=1}^d v_i = 0$. The polar polytope $(\Delta^\vee)^*$ can
be identified with the rational $d$-dimensional simplex $S_d \subset
\R_{\geq 0}^{d+1}$ in the hyperplane $ax_0 + \sum_{i=1}x_i =a+d$. The lattice vectors
 $v_0, v_1, \ldots, v_d \in N$
are exactly the generators of all $1$-dimensional cones in the $d$-dimensional
fan describing the weighted projective space $\P(a, 1^d)$ as a
$d$-dimensional toric variety.

The combinatorial structure of the pseudoreflexive polytope $\Delta$
is rather simple, because $\Delta$ is combinatorially equivalent
to the product of $(d-1)$-dimensional and $1$-dimensional simplices. Its polar
polytope $\Delta^*$ is simplicial and it has
$d+2$ vertices: the lattice vertices $v_0, v_1, \ldots v_d$ and the rational
vertex  $v_{d+1}:= - \frac{1}{b} v_0 = \frac{1}{ab}  \sum_{i=1}^d v_i$. The
vertices $v_0, v_1, \ldots, v_{d-1}$ can be chosen as a basis of the lattice $N$.

\begin{remark}
If $l =1$, i.e., $d = ab +1$ then the weighted projective space  $\P(a, 1^d)$
contains $d$  different quasi-smooth Calabi-Yau hypersurfaces $X_i \subset
\P(a, 1^d)$
$(1 \leq i \leq d)$ defined respectively by the invertible polynomials
\[ F_i(z_0, z_1, \ldots, z_d) := z_1^{a+d} + \cdots + z_d^{a+d} + z_0^{b+1} z_i,
\;1 \leq i \leq d,   \]
such that Berglund-H\"ubsch-Krawitz mirror construction can be applied
to every hypersurface $X_i$ $(1 \leq i \leq d)$ \cite{BH\"u93,Kra09}.
\end{remark}

\begin{prop} \label{l=1}
Assume that $l =1$. Then the stringy Euler number $e_{\rm str}(X)$ of a general quasi-smooth Calabi-Yau
hypersurface $X$ in $\P(a,1^d)$ equals
\[ a - \frac{1}{a} + (-1)^{d-2} \sum_{i =0}^{d-2} (-1)^i { d \choose i} (a+d)^{d-1 -i} +  (-1)^{d-1}
\sum_{i =0}^{d-1} (-1)^i { n \choose i} \frac{(a+d)^{d-i}} {a } .\]
\end{prop}

\begin{proof}
The polytope $\Delta$ is the difference of two $d$-dimensional simplices.
Therefore, all proper faces $\Theta$ of $\Delta$ are either simlices or
differences of two simplices. For any $1 \leq k \leq d-1$ there exist
exactly $2 { d \choose k-1 }$ simplicial $(d-k)$-dimensional
faces  of $\Delta$:   ${ d \choose k-1 }$ of these $(d-k)$-dimensional
simplicial faces are contained
in the hyperplane $x_0 = b+1$ and the other  ${ d \choose k-1 }$  simplicial
$(d-k)$-dimensional  faces of $\Delta$ contained in
the hyperplane $x_0 = 0$.
There exist exactly  ${ d \choose k }$  nonsimplicial
$(d-k)$-dimensional faces $\Theta$
of $\Delta$ which are differencies of two simplices.  A $(d-k)$-dimensional
face $\Theta \prec
\Delta$ is singular if and only if it is contained in the hyperplane $x_0 = b+1$,
and in this case $\Theta^*$ is a simplex with the rational vertex $v_{d+1}$ and for the  corresponding rational polytope
$\sigma^\Theta \cap \Delta^*$ one
has $v( \sigma^\Theta \cap \Delta^*) = 1/b$.  For all regular  $(d-k)$-dimensional
faces $\Theta \prec_{\rm reg}
\Delta$  one has $v( \sigma^\Theta \cap \Delta^*) = 1$.

Now we can apply the formula (\ref{E-f-u1}) for computing
the stringy Euler number of a generic Calabi-Yau
hypersurface $X \subset \P(a,1^d)$:
\[ e_{\rm st}(X)  =  (-1)^{d-1}v(\Delta) +  \sum_{k =1}^{d-1} (-1)^{d-1-k}
\sum_{ \Theta \prec \Delta \atop \dim \Theta = d-k }
 v(\Theta) \cdot v(\sigma^\Theta \cap \Delta^*) = \]
\begin{align*}
= & (-1)^{d-1} \left(\frac{(a +d)^d}{a} - \frac{1}{a} \right) + \\
+ & (-1)^{d-2} \left( { d \choose 0}\frac{1}{b} + { d \choose 0}(a +d)^{d-1} +
{ d \choose 1}
\left( \frac{(a+d)^{d-1}} {a } - \frac{1}{a} \right)  \right) +
\end{align*}
\[ \cdots
 +  (-1)^{d-1-k} \left( { d \choose k-1} \frac{1}{b} + { d \choose k-1}(a +d)^{d-k} +
{ d \choose k}
\left( \frac{(a+d)^{d-k}} {a } - \frac{1}{a} \right)  \right) + \cdots \]

\begin{align*}
  +& (-1)^{0} \left( { d \choose d-2} \frac{1}{b} + { d \choose d-2}(a+d) +
{ d \choose d-1}
\left( \frac{a+d} {a } - \frac{1}{a} \right)  \right) = \\
 = & (-1)^{d-2} \frac{1}{b} \sum_{i =0}^{d-2} (-1)^i { d \choose i} +
(-1)^{d} \frac{1}{a} \sum_{i =0}^{d-1} (-1)^i { d \choose i} + \\
 + & (-1)^{d-2} \sum_{i =0}^{\d-2} (-1)^i { d \choose i} (a+d)^{d-1 -i} +  (-1)^{d-1}
\sum_{i =0}^{d-1} (-1)^i { d \choose i} \frac{(a+d)^{d-i}} {a } =\\
= & a - \frac{1}{a} + (-1)^{d-2} \sum_{i =0}^{d-2} (-1)^i { d \choose i} (a+d)^{d-1 -i} +  (-1)^{d-1}
\sum_{i =0}^{d-1} (-1)^i { n \choose i} \frac{(a+d)^{d-i}} {a } .
\end{align*}
\end{proof}

\begin{prop} \label{e-vee}
For any $l$ $(1 \leq l \leq a-1)$ the stringy Euler number $e_{\rm str}(X^\vee)$ of
a canonical Calabi-Yau model $X^\vee$ of a $\Delta^\vee$-nondegenerated
affine hypersurface in $\T_d$ defined by the Laurent
polynomial
\[ f(x) =  x_d^{-a} \prod_{i=1}^{d-1} x_i^{-1} + \sum_{i=1}^d x_i \]
equals
\[  (-1)^{d-1} \left( a - \frac{1}{a} \right) -
  \sum_{i=1}^{d-2} (-1)^i { d \choose i} (a+d)^{d-i-1} +
\sum_{i=2}^{d-1} (-1)^i { d \choose i} \frac{(a+d)^{d-i}} {a }. \]
\end{prop}

\begin{proof}
For the dual pseudoreflexive polytope $\Delta^\vee$ one has $v(\Delta^\vee) = a +d$.
All faces $\Theta \preceq \Delta^\vee$ are lattice simplices. A $(n-k)$-dimensional
face $\Theta \prec \Delta^\vee$ is singular if and only if it is contained in the
$(d-1)$-simplex with vertices $v_1, \ldots, v_d$.

Now we apply the formula (\ref{E-f-u1}) for computing the stringy Euler number of the canonical Calabi-Yau model $X^\vee$:
 \begin{align*}
e_{\rm st}(X^\vee) = & (-1)^{d-1}v(\Delta^\vee) +  \sum_{k =1}^{d-1} (-1)^{d-1-k}
\sum_{ \Theta \prec \Delta^\vee \atop \dim \Theta = d-k }
 v(\Theta) \cdot v(\sigma^\Theta \cap (\Delta^\vee)^*) = \\ & (-1)^{d-1} (a+d) + (-1)^{d-2} \left(d + \frac{1}{a} \right) + \\
 + & (-1)^{d-3} \left( { d \choose d-2} (a+d) +
{ d \choose d-1} \frac{(a+d)} {a } \right) + \\
+ &  (-1)^{d-4} \left( { d \choose d-3 } (a+d)^2 + { d \choose d- 2} \frac{(a+d)^2} {a } \right) +  \\
 +&  (-1)^{d-5} \left( { d \choose d- 4} (a +d)^3 + { d \choose d- 3} \frac{(a+d)^3} {a } \right) +  \cdots \\
+ & (-1)^{0} \left( { d \choose 1} (a+d)^{d-2} + { d \choose 2} \frac{(a+d)^{d-2}}
{a} \right) = \\
  = & (-1)^{d-1} \left( a - \frac{1}{a} \right) - \\
- &  \sum_{i=1}^{d-2} (-1)^i { d \choose i} (a+d)^{d-i-1} +
\sum_{i=2}^{d-1} (-1)^i { d \choose i} \frac{(a+d)^{d-i}} {a }.
\end{align*}
\end{proof}

\begin{cor}
If $l=1$ then one has
\[  e_{\rm st}(X) = (-1)^{d-1}  e_{\rm st}(X^\vee). \]
\end{cor}

\begin{proof}
Comparing the formulas for $ e_{\rm st}(X)$ and  $e_{\rm st}(X^\vee)$ in \ref{l=1} and \ref{e-vee}, we see
that
\[  (-1)^{d-1}  e_{\rm st}(X^\vee) - (d+a)^{d-1} + \frac{(a +d)^{d}} {a }  -
{ d \choose 1} \frac{(d+a)^{d-1}} {a }   =  e_{\rm st}(X).  \]
Therefore, we get
\[  (-1)^{d-1}  e_{\rm st}(X^\vee)   =  e_{\rm st}(X).  \]
\end{proof}

The weighted projective space $V=\P(a,1^d)$ is a toric variety defined by a simplicial
$d$-dimensional fan whose $1$-dimensional cones
a generated by lattice vectors $v_0, v_1, \ldots, v_d$ satisfying the relation
$a v_0 + \sum_{i=1}^d v_i =0$.  It is easy to show that again
the convex hull of $\{ v_0, v_1, \ldots, v_d \}$
is a pseudoreflexive simplex $\Delta^\vee$ which is dual to $\Delta$.
There is a toric desingularization $\rho\, :\, V' \to V$
having one exceptional divisor $E \cong \P^{d-1}$ that corresponds to the lattice
point $-v_0$ so that the set of lattice vectors
$\{-v_0,  v_0, v_1, \ldots, v_d \}$ can be identified with the set of inner
normal vectors to facets of $\Delta$.   The toric desingularization $\rho \, :\,
V' \to V$ induces a desingularization  $\rho \, :\,
X' \to X$ of the generic Calabi-Yau hypersurface $X \subset V$ such that
the fiber $D:= \rho^{-1}(p)$ over the  unique singular point $x \in X$ is
isomorphic to   a generic hypersurface of degree $k$ in $\P^{d-1}$. One has
\[ K_{X'} = \rho^* K_X + bD,  \]
because the linear function $\varphi$ on the cone $\sum_{i=1}^d \R v_i$
with $\varphi(v_1) = \ldots = \varphi(v_d) =1$  has value
$b+1$ on the lattice vector $-v_0 = 1/a \sum_{i=1}^d v_i$.

\begin{theorem} \label{theo-prime}
Let $X$ be a generic $\Delta$-nondegenerated
Calabi-Yau hypersurface in the $d$-dimensional
weighted projective space $\P(a,1^d)$ $(d = ab +l, \; l \geq 2)$  and let
$X^\vee$ be a canonical Calabi-Yau model of the $\Delta^\vee$-nondegenerated
affine hypersurface $Z \subset \T_d$
defined by the Laurent polynomial
\[ f(x) = x_d^{-a} \prod_{i=1}^{d-1} x_i^{-1} + \sum_{i=1}^d x_i. \]
Then $e_{\rm str}(X) \in \frac{1}{b} \Z$ and $e_{\rm str}(X^\vee) \in \frac{1}{a} \Z$.
Moreover, if  $l=2$  then $e_{\rm str}(X)$ is not an integer. In particular,
 the equality
\[ e_{\rm str}(X) = (-1)^{d-1} e_{\rm str}(X^\vee) \]
can not be satisfied if $l =2$ and $(a, b)$ is a pair of distinct odd prime numbers.
\end{theorem}

\begin{proof}
We compute the stringy Euler number of a generic Calabi-Yau
hypersurface $X \subset \P(a, 1^d)$ as in \ref{l=1}:
\[ e_{\rm st}(X) = (-1)^{d-1} \left(\frac{(a+d)^d}{a} - \frac{l^d}{a} \right) + \]
\[ + (-1)^{d-2} \left( { n \choose 0}l^{d-1} \frac{1}{b} + { d \choose 0}(a+d)^{d-1} +
{ d \choose 1}
\left( \frac{(d+a)^{d-1}} {a } - \frac{l^{d-1}}{a} \right)  \right) \]
\[ + (-1)^{d-3} \left( { n \choose 1}  l^{d-2}\frac{1}{b} + { d \choose 1}(d+a)^{d-2} +
{ d \choose 2}
\left( \frac{(d+a)^{d-2}} {a } - \frac{l^{d-2}}{a} \right)  \right) +
 \]
\[ + (-1)^{d-4} \left( { d \choose 2}  l^{d-3}\frac{1}{b} + { d \choose 2}(d+a)^{d-3} +
{ d \choose 3}
\left( \frac{(d+a)^{d-3}} {a } - \frac{l^{d-3}}{a} \right)  \right) +
 \]
\[ \cdots \]
\[ + (-1)^{0} \left( { d \choose d-2} l \frac{1}{b} + { d \choose d-2}(d+a) +
{ d \choose d-1}
\left( \frac{d+a} {a } - \frac{l}{a} \right)  \right). \]
Since $a$ divides $(d+a)^i - l^i= (ab +a +l)^i - l^i$ for any $i \in \N$, we obtain
that $e_{\rm str}(X) \in \frac{1}{b} \Z$. The terms in $e_{\rm str}(X)$ having
the denominator $b$ sum up to
\[ A:= (-1)^{d-2} \frac{1}{lb} \sum_{i =0}^{d-2} (-1)^i { d \choose i} l^{d-i}
=  (-1)^{d-2} \frac{1}{lb} \left(  \left(  l-1 \right)^d  - (-1)^d - (-1)^{d-1} dl \right). \]
In particular, for $l =2$ and odd integers $a, b$  the dimension  $d = ab + 2$ is odd,
\[ A = \frac{ 1-d}{b} \not\in \Z \]
and  $e_{\rm str}(X)$ is not an integer.

On the other hand, we did already the computation for $e_{\rm str}(X^\vee)$ in \ref{e-vee} and
obtained
\[   e_{\rm str}(X^\vee) = (-1)^{d-1} \left( a - \frac{1}{a} \right) -
  \sum_{i=1}^{d-2} (-1)^i { d \choose i} (a+d)^{d-i-1} +
\sum_{i=2}^{d-1} (-1)^i { d \choose i} \frac{(a+d)^{d-i}} {a }.\]
This shows that $e_{\rm str}(X^\vee) \in \frac{1}{a} \Z$. Therefore, if $a$ and $b$ two
distinct odd prime numbers the
equality $e_{\rm str}(X) = (-1)^{d-1} e_{\rm str}(X^\vee)$ can hold only if the  stringy
Euler numbers are integers, but for $l=2$ this is not the case.
\end{proof}

\section{An additional condition on singular facets}

Let $\Delta \subset M_\R$ be a $d$-dimensional pseudoreflexive polytope
and let ${\Delta^\vee}:= [\Delta^*]$ be the Mavlyutov dual pseudoreflexive polytope.
We consider also two additional $d$-dimensional  almost pseudoreflexive
 polytopes $\Delta_1 \subseteq \Delta$
and $\Delta_2 \subseteq \Delta^\vee$ such that one
has the inclusions
\[ \Delta_1 \subseteq [\Delta_1^{can}] = \Delta,   \]
\[ \Delta_2 \subseteq  [\Delta_2^{can}] = \Delta^\vee . \]

A  generalization of Berglund-H\"ubsch-Krawitz mirror construction suggested by Artebani, Comparin
and Guilbot \cite{ACG16} needs an additional condition that guarantees that the Zariski
closure of an affine  $\Delta_1$-nondegenerated hypersurface $Z_1$ in the toric
variety $\P_{\Delta_2^*}$ associated with the rational polar polytope $\Delta_2^*$ will be
quasi-smooth. The same condition is demanded for the  Zariski
closure of a $\Delta_2$-nondegenerated affine hypersurface $Z_2$ in the toric
variety $\P_{\Delta_1^*}$ associated with the rational polar polytope $\Delta_1^*$.
The quasi-smoothness condition implies that the singularities of Calabi-Yau hypersurfaces
are locally quotient singularities. In particular, the stringy Euler number of such Calabi-Yau
hypersurfaces is always an integer.

In \cite[Def. 7.1.1, Prop. 7.1.3]{Bor13} Borisov suggested to generalize the quasi-smoothness
condition using  some versions of  Jacobian rings. It is not quite clear how Borisov's condition
can be described by purely combinatorial properties of convex polytopes, but it is satisfied in two cases: 1) 
for reflexive polytopes and 2) for almost pseudoreflexive simplices $\Delta_1$ and $\Delta_2$
that appear in the Berglund-H\"ubsch-Krawitz mirror construction.

Our purpose is to describe  a new  another 
condition on Calabi-Yau varieties $X$ and $X^\vee$ that must be added to 
the Mavlyutov duality for pairs of $d$-dimensional
pseudoreflexive polytopes $\Delta$ and $\Delta^\vee$
such that the stringy Euler numbers $e_{\rm str}(X)$ and $e_{\rm str}(X^\vee)$ will be
integers satisfying the equation
\[    e_{\rm str}(X) = (-1)^{d-1} e_{\rm str}(X^\vee). \]

We remark that a pseudoreflexive lattice polytope $\Delta$ is not reflexive
if and only if there exists at least one singular facet $\Theta \prec_{\rm sing} \Delta$.
Our  additional condition on a pseudoreflexive polytope $\Delta$
is exactly an additional  condition on its singular facets of $\Delta$. By \ref{reg-ord},
there exist a natural bijection between singular facets $\Theta$ of pseudoreflexive
polytope $\Delta$
and non-integral vertices $\nu_\Theta$ of the polar polytope $\Delta^*$.

Let $Z \subset \T_d$
be a $\Delta^\vee$-nondegenerated hypersurface and let $X^\vee$ be its canonical
Zariski closure in the toric $\Q$-Fano variety $\P_{\Delta^*}$ corresponding
to the polar polytope $\Delta^*$. Then for any singular
facet $\Theta \prec_{\rm sing} \Delta$ the Calabi-Yau hypersurface $X^\vee \subset \P_{\Delta^*}$
contains the torus fixed
point $x_\Theta \in X^\vee$ corresponding to the rational vertex $\nu_\Theta \in {\Delta^*}$.

\begin{definition}
Let $\Theta \prec_{\rm sing} \Delta$  be a singular facet of a $d$-dimensional pseudoreflexive polytope $\Delta$. Denote by $n_\Theta$ $(n_\Theta \geq 2)$ the integral distance from $0 \in M$ to the facet $\Theta$. We call the facet $\Theta$ {\em quasi-regular} if the local stringy Euler
number $e_{\rm str}(X^\vee, x_\Theta)$ is an integer that can
be computed by the formula:
\[e_{\rm str}(X^\vee, x_\Theta) =
n_\Theta \cdot v(\Theta).\]
\end{definition}

We illustrate this definition with the examples of the Mavlyutov pairs $(\Delta, \Delta^\vee)$
from the previous section.

\begin{example}
We consider  $\Delta^\vee$ to be a $d$-dimensional pseudoreflexive simplex  which is the convex
hull of a basis $e_1, \ldots, e_d$ of the lattice $M$ and a point $e_0= - a e_d + \sum_{i=1}^{d-1}$, where
$a$ does not divide $d$ and $a \leq d/2$. Let $d = ab +l$ for some integers $1 \leq l < a$ and $b \geq 2$. Then  the dual pseudoreflexive polytope $\Delta =
(\Delta^\vee)^\vee$ is the Newton polytope of a Calabi-Yau hypersurface $X $ of
degree $a + d$ in the $d$-dimensional weighted projective space $\P(a, 1^d)$ that contains
a torus fixed point $x:= (1:0: \ldots :0)$. A desingularization  of $\P(a, 1^d)$  at $x$
contains a single exceptional divisor isomorphic to $\P^{d-1}$. The induced birational
morphism $\rho\, :\, Y \to X$ has a single exceptional divisor $D \subset \P^{d-1}$
which is a hypersurface of degree $l$ and one has $K_Y = \rho^* K_{X} + (b-1)D$. By \ref{isol-sing}, we obtain $e_{\rm str}(X, x) = e(D)/b$. On the other hand, we have $n_\Theta \cdot v(\Theta) = a$,
where $\Theta$ is a singular facet with vertices
$v_1, \ldots, v_d$ of $\Delta^\vee$ corresponding to the point $x$. The equality
$ e_{\rm str}(X, x)  = n_\Theta \cdot v(\Theta) $ is equivalent to $e(D) = ab = d-l$. This
can happen for a smooth $(d-2)$-dimensional hypersurface $D$ of degree $l$ in $\P^{d-1}$
if and only if $l =1$, i.e., only if $X$ is a quasi-smooth hypersurface in  $\P(a, 1^d)$.
 \end{example}

\begin{theorem}\label{theo-cond}
Let  $(\Delta, {\Delta^\vee})$ be a Mavlyutov pair
of  $d$-dimensional pseudoreflexive polytopes $\Delta$ and
${\Delta^\vee}$.
Assume all singular facets $\Theta'\prec_{\rm sing} \Delta^\vee$
 are quasi-regular.
 Then the stringy Euler number $e_{\rm str}(X)$ of a
canonical Calabi-Yau model $X$ of a  $\Delta$-nondegenerate
hypersurface can be computed
by the following formula:
\[   \sum_{ \Theta' \prec_{\rm sing}  {\Delta^\vee} \atop
\dim \Theta' = d-1 }  n_{\Theta'} \cdot v( \Theta')
+   \sum_{  \Theta \prec_{\rm reg}   \Delta \atop  1 \leq
\dim \Theta \leq d-2}   (-1)^{\dim \Theta -1}
 v(\Theta) \cdot v(
\Theta^\vee) + (-1)^{d-1} \sum_{ \Theta \prec_{\rm sing}  {\Delta} \atop
\dim \theta = d-1 }  n_\Theta \cdot v( \Theta) . \]
In particular, if all singular facets of $\Delta$ are also quasi-regular
then for canonical Calabi-Yau models  $X^\vee$  of a
$\Delta^\vee$-nondegenerate hypersuface  one obtains the equality
\[  e_{\rm str}(X)  = (-1)^{d-1}  e_{\rm str}(X^\vee). \]
\end{theorem}

\begin{proof}
Let $Z \subset \T_d$ be a $\Delta$-nondegenerate affine hypersurface. There
are two projective closures of $Z$: the closure $\overline{Z}$ in the toric
variety $\P_\Delta$ and the canonical model $X$
obtained as the
Zariski closure of $Z$ in the  toric $\Q$-Fano
variety corresponding to the rational polytope $\Delta^{can}=
({\Delta^\vee})^* \subset M_\R$.

We choose a regular simplicial
 fan $\widehat{\Sigma}$ which is  a common subdivision of
two rational polyhedral fans:  the normal fan $\Sigma^\Delta$ and
the normal fan $\Sigma^{\Delta^{can}}$. So we obtain
two birational toric morhisms $\rho_1$ and $\rho_2$:
\[
\xymatrix{   & \P_{\widehat{\Sigma}}\ar[ld]_{\rho_1} \ar[rd]^{\rho_2} & \\
\P_{\Delta}  \ar@{-->}[rr]^f &  & \P_{\Delta^{can}}}
\]
together with the induced birational morphisms
\[
\xymatrix{   & \widehat{Z}\ar[ld]_{\rho_1} \ar[rd]^{\rho_2} & \\
\overline{Z} \ar@{-->}[rr]^f &  & X}
\]

The canonical Calabi-Yau hypersurface  $X \subset \P_{\Delta^{can}}$ is a disjoint
union of locally closed strata
$X_F := X \cap \T_F$
where $\T_F$ is a torus orbit in the projective toric variety
$\P_{\Delta^{can}}$ and $F$ runs over all
faces $F \preceq \Delta^{can}$ of the rational polytope $\Delta^{can}$:
 \[ X = \bigcup_{  F \preceq \Delta^{can}} X_F. \]

Let $v_1, \ldots, v_s$ be the set of primitive lattice generators of $1$-dimensional cones
in the fan $\widehat{\Sigma}$. We set $I := \{1, \ldots, s \}$. Then $k$-dimensional cone $\sigma \in \widehat{\Sigma}$ is  determined by a subset $J \subset I$ such that $|J| =k$ and $\sigma$ is generated
by $v_j $ $(j \in J)$.

For any face $F \preceq \Delta^{can}$ we define the stringy Euler number
\[ e_{\rm str}(X, X_F) := \sum_{\emptyset \subseteq J \subseteq I} e(D_J^\circ \cap \rho_2^{-1}(\T_F))
\prod_{j \in J} \frac{1}{a_j + 1}, \]
where $D_J^\circ$ are either empty or a locally closed stratum  on the smooth projective hypersurface $\widehat{Z}$ in the
toric variety $\P_{\widehat{\Sigma}}$ corresponding to a cone $\sigma \in \widehat{\Sigma}$
of dimension $|J|$. By the additivity of the Euler number, we obtain
\[ e_{\rm str}(X) = \sum_{ F \preceq \Delta^{can}} e_{\rm str}(X, X_F). \]
So it remains to compute $e_{\rm str}(X, X_F)$ for any face $F \preceq \Delta^{can}$.

We consider the following $4$ possibilities for a face $F \preceq \Delta^{can}$:

\begin{itemize}

\item $\dim [F] = \dim F = k \geq 1$, i. e., $F= \Theta^*$ for some regular
$(d-k-1)$-dimensional face $\Theta \preceq \Delta^\vee$

\item $\dim [F] <  \dim F = k \geq 1$, i. e., $F= \Theta^*$ for some singular
$(d-k-1)$-dimensional face $\Theta \preceq \Delta^\vee$

\item  $\dim [F] = \dim F = 0$, i. e., $F= \Theta^*$ is a lattice vertex of
$\Delta^{can}$ corresponding to some regular
$(d-1)$-dimensional face $\Theta \preceq \Delta^\vee$

\item $\dim F = 0$ and $[F] = \emptyset$, i. e., $F= \Theta^*$ is a rational vertex of
$\Delta^{can}$ corresponding to some singular
$(d-1)$-dimensional face $\Theta \preceq \Delta^\vee$.
\end{itemize}

If $\dim [F] = \dim F = k \geq 1$, then $[F] = \Theta$ for some $k$-dimensional face $\Theta \preceq \Delta$.  For a generic $\Delta$-nondegenerate hypersurface $Z$ the affine hypersuface
$X_F  \subset \T_F$ is $[F]$-nondegenerate
and its  Euler number equals $(-1)^{k-1}v([F])$ (see \ref{kh78}). Moreover,
$X$ has Gorenstein toroidal singularities along $X_F$ corresponding to the $(d-k)$-dimensional
cone over the dual regular $(d-k-1)$-dimensional face $|Theta^\vee$ of $\Delta$. So one has
\[ e_{\rm str}(X, X_F) = (-1)^{k-1} v(\Theta) \cdot v(\Theta^\vee). \]

If $\dim [F] <  \dim F = k \geq 1$, then the affine hypersuface
$X_F  \subset \T_F$ is isomophic to a product of $(\C^*)^{k - \dim [F]}$ and $[F]$-nondegenerated
affine hypersurface.  Therefore $e(X_F)=0$ and one has $e_{\rm str}(X, X_F) = 0$.

If $\dim [F] = \dim F = 0$, then $X_F$ is empty and one has $e_{\rm str}(X, X_F) = 0$.

If $\dim F = 0$ and $[F] = \emptyset$, then $X_F$ is a torus fixed point $x_{\Theta'} \in \P_{\Delta^{can}}$
$e_{\rm str}(X, X_F)$ equals to local stringy Euler number $e_{\rm str}(X, x_\Theta) =
n(\Theta) \cdot v(\Theta)$ for some singular facet $\Theta \prec_{\rm sing} \Delta^\vee$.

Thus, we obtain
\[  e_{\rm str}(X) = \sum_{ \Theta_{\rm sing} \prec  {\Delta^\vee} \atop
\dim \Theta_{\rm sing} = d-1 }  n_\Theta \cdot v( \Theta)
+   \sum_{k =1}^d (-1)^{k-1} \sum_{  \Theta_{\rm ord} \prec  \Delta \atop
\dim \Theta_{\rm ord} = k }
 v(\Theta_{\rm ord}) \cdot v(
\Theta^*_{\rm ord}). \]

Since  the $d$-dimensional lattice polytope $\Delta$ is
the union over all $(d-1)$-dimensional
faces $\Theta \prec \Delta$  of $d$-dimensional pyramids $\Pi_\Theta:= {\rm Conv}(0, \Theta)$
with vertex $0$,  one has
\[ v(\Delta) = \sum_{  \Theta \prec \Delta \atop
\dim \Theta = d-1}  v(\Pi_\Theta).  \]
On the other hand,
$v(\Pi_\Theta) = v(\Theta) \cdot
n_\Theta$,  where $n_\Theta$ is the integral distance from $\Theta$ to
$0 \in M$.  The equality  $n_\Theta =1$ holds if and only if  $\Theta \prec \Delta $ is a regular
 $(d-1)$-dimensional face.  This implies the equality
\begin{align*} \sum_{k =d-1}^{d} (-1)^{k-1} \sum_{ \substack{ \Theta \prec_{\rm reg}
\Delta \\
\dim \Theta = k} }
 v(\Theta) \cdot v(
\Theta^\vee)
 = & (-1)^{d-1} \left( v(\Delta) -
\sum_{  \Theta \prec_{\rm reg}  \Delta \atop
\dim \Theta = d-1}  v(\Theta) \right)  = \\
= &  (-1)^{d-1} \left( \sum_{  \Theta \prec_{\rm sing}  \Delta \atop
\dim \Theta = d-1}  v(\Theta) \cdot n_\Theta  \right)
\end{align*}
that proves the required formula for $e_{\rm str}(X)$.

The equality  $e_{\rm str}(X)  = (-1)^{d-1}  e_{\rm str}(X^\vee)$ follows now from the
duality  $\Theta \leftrightarrow \Theta^\vee$
between $k$-dimensional regular  faces $\Theta \prec_{\rm reg}  \Delta$ and
$(d-k-1)$-dimensional regular faces
$\Theta^\vee \prec_{\rm reg} {\Delta^\vee}$ and from
the equality
\[ \sum_{k =1}^{d-2} (-1)^{k-1} \sum_{ \substack{ \Theta \prec_{\rm reg}  \Delta \\
\dim \Theta = k} }
 v(\Theta) \cdot v(
\Theta^\vee) = (-1)^{d-1}  \left(
\sum_{k =1}^{d-2} (-1)^{k-1} \sum_{ \substack{ \Theta^\vee \prec_{\rm reg}  \Delta^\vee \\
\dim \Theta^\vee = k} } v(\Theta) \cdot v(
\Theta^\vee) \right)  \]
\end{proof}


\end{document}